\documentclass[reqno,12pt]{amsart}
\usepackage{amsfonts}
\usepackage{bbm}
\usepackage{} 
\numberwithin{equation}{section}
\usepackage{indentfirst}
 \usepackage{color}
\usepackage{amssymb}
\usepackage{mathrsfs}
\usepackage{xy}
\xyoption{all}
\def\Ext{\mbox{\rm Ext}\,} \def\Hom{\mbox{\rm Hom}} \def\dim{\mbox{\rm dim}\,} \def\Iso{\mbox{\rm Iso}\,}\def\Ind{\mbox{\rm Ind}\,}
\def\lr#1{\langle #1\rangle}    
\def\Ker{\mbox{\rm Ker}\,}   \def\im{\mbox{\rm Im}\,} \def\Coker{\mbox{\rm Coker}\,}
\def\End{\mbox{\rm End}\,}\def\tw{\mbox{\rm tw}\,}\def\id{\mbox{\rm id}\,}

\def\rad{\mbox{\rm rad}\,}\def\M{\mathcal{M}}\def\grad{\mbox{\rm grad}\,}
\def\Aut{\mbox{\rm Aut}\,}\def\Dim{\mbox{\rm \textbf{dim}}\,}\def\A{\mathcal{A}\,} \def\H{\mathcal{H}\,}
\def\P{\mathscr{P}\,}\def\X{\mathbb X}\def\T{\mathbb T}\def\e{\mbox{\rm e}}

\def\ZZ{\mathbb Z}

\theoremstyle{plain} 
\newtheorem{theorem}{\bf Theorem}[section]
\newtheorem{lemma}[theorem]{\bf Lemma}
\newtheorem{corollary}[theorem]{\bf Corollary}
\newtheorem{proposition}[theorem]{\bf Proposition}

\theoremstyle{definition} 
\newtheorem{definition}[theorem]{\bf Definition}
\newtheorem{remark}[theorem]{\bf Remark}
\newtheorem{example}[theorem]{\bf Example}

\newcommand{\bt}{\begin{theorem}}
\newcommand{\et}{\end{theorem}}
\newcommand{\bl}{\begin{lemma}}
\newcommand{\el}{\end{lemma}}
\newcommand{\bd}{\begin{definition}}
\newcommand{\ed}{\end{definition}}
\newcommand{\bc}{\begin{corollary}}
\newcommand{\ec}{\end{corollary}}
\newcommand{\bp}{\begin{proof}}
\newcommand{\ep}{\end{proof}}
\newcommand{\bx}{\begin{example}}
\newcommand{\ex}{\end{example}}
\newcommand{\br}{\begin{remark}}
\newcommand{\er}{\end{remark}}
\newcommand{\be}{\begin{equation}}
\newcommand{\ee}{\end{equation}}
\newcommand{\ba}{\begin{align}}
\newcommand{\ea}{\end{align}}
\newcommand{\bn}{\begin{enumerate}}
\newcommand{\en}{\end{enumerate}}
\newcommand{\bcs}{\begin{cases}}
\newcommand{\ecs}{\end{cases}}

\makeatletter
\renewcommand{\section}{\@startsection{section}{1}{0mm}
  {-\baselineskip}{0.5\baselineskip}{\bf\leftline}}
\makeatother

\begin{document}

\title[Acyclic quantum cluster algebras via Hall algebras of morphisms]{Acyclic quantum cluster algebras\\ via Hall algebras of morphisms} 

\author{Ming Ding, Fan Xu and Haicheng Zhang$^{*}$}
\address{School of Mathematics and Information Science\\
Guangzhou University, Guangzhou 510006, P.~R.~China}
\email{m-ding04@mails.tsinghua.edu.cn (M. Ding)}
\address{Department of Mathematical Sciences\\
Tsinghua University\\
Beijing 100084, P.~R.~China} \email{fanxu@mail.tsinghua.edu.cn
(F. Xu)
}
\address{Institute of Mathematics, School of Mathematical Sciences, Nanjing Normal University,
Nanjing 210023, P.~R.~China}
\email{zhanghc@njnu.edu.cn (H. Zhang)}

\thanks{$*$~Corresponding author.}


\begin{abstract}
Let $A$ be the path algebra of a finite acyclic quiver $Q$ over a finite field.
We realize the quantum cluster algebra with principal coefficients associated to $Q$ as a sub-quotient of a certain Hall algebra involving the category of morphisms between projective $A$-modules.
\end{abstract}

\maketitle

\section{Introduction}
The Hall algebra of a finite dimensional algebra $A$ over a finite field was introduced by Ringel \cite{R90} in 1990.
Ringel \cite{R90,R90a} proved that if $A$ is a representation-finite hereditary algebra, the Ringel--Hall algebra of $A$ provides a realization of the positive part of the corresponding quantum group.
Ringel's approach establishes a relation between the representation theory of algebras and Lie theory, and provides an algebraic framework for studying the Lie theory resulting from Hall algebras associated to various abelian categories. To\"en \cite{Toen2006} generalized Ringel's construction to define the derived Hall algebra for a DG-enhanced triangulated category satisfying certain finiteness conditions. Later on, for a triangulated category satisfying the left homological finiteness condition, Xiao and Xu \cite{XiaoXu} showed that To\"en's construction still provides an associative unital algebra. It was expected but so far not successful to realize an entire quantum group through derived Hall algebra over triangulated category. In 2013,
Bridgeland \cite{Bri13} provided a realization of the
whole quantum group via the Hall algebra of 2-cyclic complexes of projective modules over a hereditary algebra.

Lusztig \cite{Lu90,Lu91} invented the geometric version of Ringel--Hall algebra constructions and obtained the canonical basis of the positive part of a quantum group as the direct summands of some semisimple constructible complexes over module varieties of a quiver. Kashiwara \cite{Ka91} applied an algebraic approach to define the crystal basis of the positive part of a quantum group. It is noteworthy that the canonical basis of a quantum group coincides with its crystal basis. In \cite{L3}, Lusztig has also introduced the semicanonical basis of the positive part of the enveloping algebra associated to a quiver $Q$  as certain constructible functions over the varieties of nilpotent representations of the preprojective algebra of $Q$, and the basis is indexed by the irreducible components of the varieties.

Cluster algebras were introduced by Fomin and Zelevinsky in \cite{FZ} and later the quantum cluster algebras were introduced by Berenstein and Zelevinsky in \cite{BZ05}. Inventions of cluster algebras and quantum cluster algebras are aimed to develop a combinatorial approach to the dual semicanonical bases in coordinate rings and dual canonical bases in quantum deformations of varieties related to algebraic groups. Geiss, Leclerc, and Schr\"oer proved in
\cite{GLS} that the cluster monomials of certain cluster algebras are elements of the dual of Lusztig's semicanonical basis. Kimura and Qin \cite{KQ} categorified the quantum cluster algebra of an acyclic quiver via the generalized graded quiver variety, and as a byproduct it is proved that the quantum cluster monomials belong to the corresponding dual canonical basis. Moreover, Qin \cite{Qin2} extended these results to quantum cluster algebras which are injective-reachable.

A natural idea  is to construct a framework to explicitly relate Hall algebras with (quantum) cluster algebras. In \cite{CK2005}, Caldero and Keller suggested the similarity between the multiplication in a cluster algebra and that in a dual Hall algebra. In \cite[Theorem 3.3]{DX}, the similarity was confirmed for the quantum cluster algebra of an acyclic quiver (see \cite{BR} for the generalization). Given a finite acyclic quiver $Q$, let $\mathcal{AH}_q(Q)$ be the subalgebra of a certain skew-field of fractions generated by quantum cluster characters (see Section \ref{quantum_cluster_char} for more details). Then, there exists an algebra homomorphism from the dual Hall algebra associated to the representation category of $Q$ to $\mathcal{AH}_q(Q)$ (cf. \cite{CDX}). It is pitiful that this homorphism may not be surjective, in particular, there might be no preimages of  the initial quantum cluster variables.

The aim of this paper is to construct a surjective algebra homomorphism from a certain Hall algebra to  $\mathcal{AH}_q(Q)$, and then realize the quantum cluster algebra as a sub-quotient algebra of this Hall algebra. In order to achieve this aim, we construct the localized Hall algebra $\mathcal{MH}(Q)$ associated to the morphism category $\mathcal{C}_2(\mathcal{P})$ (see Section \ref{morphism} for the definition), which has indecomposable objects indexed by the isomorphism classes of indecomposable objects in the cluster category of $Q$. The algebra $\mathcal{MH}(Q)$ contains the dual Hall algebra of $Q$ as a subalgebra and is isomorphic to the subalgebra of the extended dual derived Hall algebra of $Q$ generated by all objects ${M\oplus P[1]}$ corresponding to $kQ$-modules $M$ and projective $kQ$-modules $P$ (see Theorem \ref{mor_derived}).

The main result of this paper is Theorem \ref{alg-homo} which gives a surjective homorphism from a twisted version of $\mathcal{MH}(Q)$ to $\mathcal{AH}_q(Q)$.
This surjective algebra homomorphism will motivate the interactions between Hall algebras and quantum cluster algebras. For example, one can compare the categorifications of quantum groups and quantum cluster algebras, and then compare their dual canonical bases.

The paper is organized as follows: In Section 2 we mainly summarize some homological properties of a morphism category, and then define the associated Hall algebra via Bridgeland's approach in Section 3. The characterizations of the multiplicative structure of thus defined Hall algebra are given in Sections 4 and 5. Section 6 is devoted to establishing the relation between this Hall algebra and the extended dual derived Hall algebra. We reformulate the definitions of quantum cluster characters and give the corresponding cluster multiplication formulas in Section 7. Finally, we prove the main theorem of this paper in Section 8.

Let us fix some notations used throughout the paper. Let $k=\mathbb{F}_q$ be always a finite field with
$q$ elements, and $\ZZ[{q}^{\pm\frac{1}{2}}]$ be the ring of integral Laurent polynomials. Let $\A$ be an (essentially small) finitary hereditary abelian $k$-category with enough projectives, where the term ``finitary" means that for any objects $M,N\in\A$, $\Hom_{\A}(M,N)$ and $\Ext^1_{\A}(M,N)$ are both finite dimensional. Let $\mathscr{P}\subset\A$ be the subcategory consisting of projective objects. Denote by $C^b(\A)$ and $D^b(\A)$ the category of bounded complexes over $\A$ and its bounded derived category, respectively. For each $M_\bullet\in C^b(\A)$, its $i$-th homology is denoted by $H_i(M_\bullet)$.  The Grothendieck group of $\A$ and the set of isomorphism classes $[X]$ of objects in $\A$ are denoted by $K(\A)$ and $\Iso(\A)$, respectively. For any object $M\in\A$ we denote by $\hat{M}$ or $\Dim M$ the image of $M$ in $K(\A)$. For a finite set $S$, we denote by $|S|$ its cardinality. For an object $M$ in an additive category, we denote by $\Aut(M)$ the automorphism group of $M$, and set $a_M:=|\Aut(M)|$. We always assume that all the vectors are column vectors.

\section{Preliminaries}

\subsection{Hall algebras}
Given objects $L,M,N \in \mathcal{A}$, let $\Ext_\mathcal{A}^1(M,N)_L \subset \Ext_\mathcal{A}^1(M,N)$ be the subset consisting of those equivalence classes of short exact sequences with middle term isomorphic to $L$.
\begin{definition}\label{Hall algebra of abelian category}
The \emph{Hall algebra} $\mathcal {H}(\mathcal{A})$ of $\mathcal{A}$ is the free $\ZZ[{q}^{\pm\frac{1}{2}}]$-module with basis elements $[M] \in \Iso(\mathcal{A})$, and with the multiplication defined by
\[[M] \diamond [N] = \sum\limits_{[L] \in \Iso(\mathcal{A})} {\frac{{|\Ext_\mathcal{A}^1{{(M,N)}_L}|}}{{|\Hom_\mathcal{A}(M,N)|}}} [L].\]
\end{definition}
\begin{remark}
Given objects $L,M,N\in \A$, set
$$F_{MN}^L:=|\{N'\subset L~|~N'\cong N, L/N'\cong M\}|.$$
By the Riedtmann--Peng formula \cite{Riedtmann,Peng},
$$F_{MN}^{L}=\frac{|\Ext^1_{\A}(M,N)_{L}|}{|\Hom_{\A}(M,N)|}\frac{a_{L}}{a_{M}a_{N}}.$$ Thus in terms of alternative generators $[[M]]=\frac{[M]}{a_M}$, the product takes the form
$$[[M]]\diamond [[N]]= \sum\limits_{[L] \in \Iso(\mathcal{A})}F_{MN}^L[[L]],$$
which is the definition used, for example, in \cite{R90a,Sc}. The associativity of Hall algebras amounts to the following identity
\begin{equation}\label{jiehe}\sum\limits_{[M]}F_{XY}^MF_{MZ}^L=\sum\limits_{[N] }F_{XN}^LF_{YZ}^N,\end{equation} for any objects $L,X,Y,Z\in\A$.
\end{remark}

Given objects $M,N \in \mathcal{A}$, one defines \begin{equation}\label{Euler form}\lr{M,N}:=\dim_k\Hom_{\A}(M,N)-\dim_k\Ext^1_{\A}(M,N),\end{equation}
which descends to give a bilinear form
$$\lr{\cdot ,\cdot }: K(\mathcal{A})\times K(\mathcal{A})\longrightarrow \mathbb{Z},$$
called the \emph{Euler form} of $\mathcal{A}$.

The \emph{twisted Hall algebra} $\H_{\tw}(\A)$ is the same module as $\H(\A)$ but with the twisted multiplication defined by
$$[M]*[N]=q^{\lr{M,N}}[M]\diamond[N].$$

\subsection{Morphism categories}\label{morphism}
Let $C_2(\A)$ be the category whose objects are morphisms $\xymatrix{M_{-1}\ar[r]^f&M_0}$ in $\A$, and each morphism from $\xymatrix{M_{-1}\ar[r]^f&M_0}$ to $\xymatrix{N_{-1}\ar[r]^g&N_0}$ is a pair $(u,v)$ of morphisms in $\A$ such that the following diagram
$$\xymatrix{M_{-1}\ar[r]^f\ar[d]_u&M_0\ar[d]^v\\
N_{-1}\ar[r]^g&N_0}$$ is commutative. Clearly, we may consider $C_2(\A)$ as an extension-closed subcategory of $C^b(\A)$ by identifying each morphism $\xymatrix{M_{-1}\ar[r]^f&M_0}$ with the complex $M_\bullet$ whose components of degrees $-1$ and $0$ are $M_{-1}$ and $M_0$, respectively, and other components are zero. In what follows, we also write $M_\bullet$ as a morphism $\xymatrix{M_{-1}\ar[r]^f&M_0.}$ Let $C_2(\P)\subset C_2(\A)$ be the extension-closed subcategory consisting of morphisms in $\P$. Denote by $C^b(\P)\subset C^b(\A)$ the subcategory consisting of bounded complexes over projectives, and denote by $K^b(\P)$ its homotopy category. Since $\A$ is hereditary, it is well known that $K^b(\P)$ is equivalent to $D^b(\A)$ as triangulated categories. For any objects $M_\bullet$ and $N_\bullet$ in $C_2(\P)$, by \cite[Lemma 3.1]{Gor2}, we know that
\begin{equation}
\begin{split}
\Ext^1_{C_2(\P)}(M_\bullet,N_\bullet)&\cong\Ext^1_{C^b(\P)}(M_\bullet,N_\bullet)\\
&\cong\Hom_{K^b(\P)}(M_\bullet,N_\bullet[1])\\
&\cong\Hom_{D^b(\A)}(M_\bullet,N_\bullet[1]).
\end{split}
\end{equation}

For each object $P\in\P$, define two objects in $C_2(\P)$
\begin{equation}K_P:=\xymatrix{P\ar[r]^1&P}~~\text{and}~~Z_P:=\xymatrix{P\ar[r]&0.}\end{equation}

We have the following well known result:
\begin{lemma}\cite[Lemma 4.1]{Bri13}\label{jx}
Given $M\in\A$, each projective resolution of $M$ is isomorphic to a resolution of the form
\begin{equation*}
\xymatrix{0\ar[r]&\Omega_M\oplus R\ar[r]^{\delta_M\oplus1}&P_M\oplus R\ar[r]&M\ar[r]&0,}
\end{equation*}
for some $R\in\P$ and some minimal projective resolution\footnote{The notations $P_M$ and $\Omega_M$ will be used throughout the paper.}
\begin{equation}\label{mpr}
\xymatrix{0\ar[r]&\Omega_M\ar[r]^{\delta_M}&P_M\ar[r]&M\ar[r]&0.}
\end{equation}
\end{lemma}

Given an object $M\in\A$, take a minimal projective resolution (\ref{mpr}). Then we define an object in $C_2(\P)$
\begin{equation}\label{cm}
C_M:=\xymatrix{\Omega_M\ar[r]^{\delta_M}&P_M.}
\end{equation}
By Lemma \ref{jx}, we know that any two minimal projective resolutions of $M$ are isomorphic, so $C_M$ is well defined up to isomorphism.

Now we describe the indecomposable objects in $C_2(\P)$ in the following
\begin{proposition}\label{fenjie}
Each object $M_\bullet$ in $C_2(\P)$ has a direct sum decomposition
$$M_\bullet=K_P\oplus Z_Q\oplus C_M$$
for some $P, Q\in\P$ and $M\in\A$. Moreover, the objects $P, Q$ and $M$ are uniquely determined up to isomorphism.
\end{proposition}
\begin{proof}
Clearly, $C_2(\A)$ is an abelian category, and thus idempotents split in $C_2(\A)$. For any $Z_\bullet\in C_2(\P)$ and idempotent $e:Z_\bullet\to Z_\bullet$, we obtain in $C_2(\A)$ that $Z_\bullet=X_\bullet\oplus Y_\bullet$. It is clear that
$X_\bullet$ and $Y_\bullet$ are in $C_2(\P)$. That is, idempotents split in $C_2(\P)$.
Moreover, since $C_2(\P)$ is Hom-finite, we conclude that $C_2(\P)$ is a Krull--Schmidt category. Now,
we assume that $M_\bullet=\xymatrix{M_{-1}\ar[r]^f&M_0}$ is a nonzero indecomposable object.

If $\Ker(f)=0$, then we have a short exact sequence
$$\xymatrix{0\ar[r]&M_{-1}\ar[r]^f&M_0\ar[r]&M\ar[r]&0,}$$ where $M=\Coker(f)$. By Lemma \ref{jx}, $M_\bullet\cong C_M\oplus K_P$ for some $P\in\P$. Since $M_\bullet$ is indecomposable, $M_\bullet\cong C_M$ if $M\neq0$; otherwise, $M_\bullet\cong K_P$.

If $\Ker(f)\neq0$, consider the short exact sequence
\begin{equation}\label{KI}\xymatrix{0\ar[r]&\Ker(f)\ar[r]^i&M_{-1}\ar[r]^{\pi}&\im(f)\ar[r]&0.}\end{equation}
Since $\A$ is hereditary, we know that $\im(f)$ is projective, and thus the sequence (\ref{KI}) is splitting. That is, there exists a morphism $r:M_{-1}\to\Ker(f)$ such that $r\circ i=\id.$ Then we have morphisms in $C_2(\P)$
\begin{equation*}\xymatrix{\Ker(f)\ar[r]\ar[d]^i&0\ar[d]\\
M_{-1}\ar[r]^f&M_0,}\quad\quad \xymatrix{M_{-1}\ar[r]^f\ar[d]^r&M_0\ar[d]\\
\Ker(f)\ar[r]&0.}\end{equation*}
Hence $Z_Q=\xymatrix{Q\ar[r]&0,}$ where $Q=\Ker(f)$, is a nonzero direct summand of $M_\bullet$. Since $M_\bullet$ is indecomposable, we obtain that $M_\bullet\cong Z_{Q}$.
\end{proof}

\begin{lemma}\label{hom}
For any $M_\bullet\in C_2(\P)$, $M,N\in\A$ and $P\in\P$, we have that
\begin{flalign}
&|\Hom_{C_2(\P)}(K_P,M_\bullet)|=|\Hom_{\A}(P,M_{-1})|;\label{v_1}\\
&|\Hom_{C_2(\P)}(M_\bullet,K_P)|=|\Hom_{\A}(M_0,P)|;\\
&|\Hom_{C_2(\P)}(C_P,M_\bullet)|=|\Hom_{\A}(P,M_0)|;\\
&|\Hom_{C_2(\P)}(M_\bullet,Z_P)|=|\Hom_{\A}(M_{-1},P)|;\label{v_4}\\
&|\Hom_{C_2(\P)}(C_M,C_N)|=|\Hom_{\A}(M,N)|\cdot|\Hom_{\A}(P_M,\Omega_N)|.\label{v_5}
\end{flalign}
\end{lemma}
\begin{proof}
The identities in $(\ref{v_1}$-$\ref{v_4})$ are taken from \cite[Proposition 3.1]{Bau}, and they can also be easily obtained by direct calculations. We only prove the identity in (\ref{v_5}).

By the comparison lemma in homological algebra, it is easy to get a surjective map
$$\xymatrix{\varphi:\Hom_{C_2(\P)}(C_M,C_N)\ar@{->>}[r]& \Hom_{\A}(M,N).}$$
Then by the universal property of kernels, we can obtain $\Ker(\varphi)\cong\Hom_{\A}(P_M,\Omega_N)$.
\end{proof}

\begin{lemma}\cite[Corollaries 3.1 and 3.2]{Bau}\label{proinj}
The objects $C_P$ and $K_P$, where $P\in\P$ is indecomposable, provide a complete set of indecomposable projective objects in $C_2(\mathscr{P})$; and the objects $Z_P$ and $K_P$ provide a complete set of indecomposable injective objects. Moreover, all $K_P$ are exactly the whole indecomposable projective-injective objects.
\end{lemma}

The existence of almost split sequences in $C_2(\P)$ was studied in \cite{Bau}, see also \cite{Bau2,Cha}. Let us give an example of the Auslander--Reiten quiver of $C_2(\P)$ as follows:
\begin{example}\cite[Example 6.7]{Cha}
Let $\A$ be the category of finite dimensional representations of the quiver
$$1\longrightarrow2\longrightarrow3.$$
Then the Auslander--Reiten quiver of $C_2(\P)$ is the following
$$
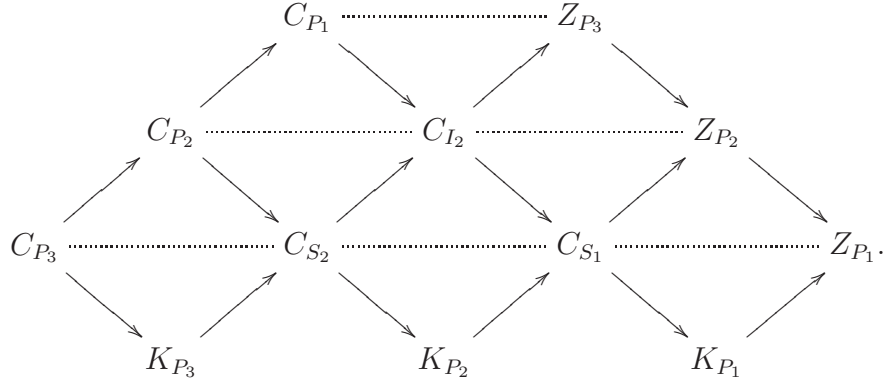
\xymatrix{&&C_{P_1}\ar[rd]\ar@{.}[rr]&&Z_{P_3}\ar[rd]&&\\&C_{P_2}\ar@{.}[rr]\ar[ru]\ar[rd]&&C_{I_2}\ar@{.}[rr]\ar[ru]\ar[rd]&&Z_{P_2}\ar[rd]&\\C_{P_3}\ar[ru]\ar[rd]\ar@{.}[rr]&&C_{S_2}\ar[ru]\ar[rd]\ar@{.}[rr]&&C_{S_1}\ar[ru]\ar[rd]\ar@{.}[rr]&&Z_{P_1}.\\&K_{P_3}\ar[ru]&&K_{P_2}\ar[ru]&&K_{P_1}\ar[ru]&}$$
\end{example}

For each $M_\bullet\in C_2(\P)$, we have a projective resolution and an injective resolution of $M_\bullet$ in the following
\begin{lemma}\cite[Proposition 3.2]{Bau}\label{resol}
For each $M_\bullet\in C_2(\P)$, we have the following short exact sequences
\begin{flalign}
&0\longrightarrow C_{M_{-1}}\longrightarrow C_{M_0}\oplus K_{M_{-1}}\longrightarrow M_{\bullet}\longrightarrow0;\label{tfj}\\
&0\longrightarrow M_{\bullet}\longrightarrow Z_{M_{-1}}\oplus K_{M_0}\longrightarrow Z_{M_0}\longrightarrow0.
\end{flalign}
\end{lemma}

By Lemma \ref{resol}, we know that the global dimension of $C_2(\P)$ is equal to one. Similarly to the Euler form of $\A$ defined in Section 2.1, for any objects $M_\bullet,N_\bullet\in C_2(\P)$, we define
\begin{equation}
\lr{M_\bullet,N_\bullet}:=\dim_k\Hom_{C_2(\P)}(M_\bullet,N_\bullet)-\dim_k\Ext_{C_2(\P)}^1(M_\bullet,N_\bullet).
\end{equation}
It also induces a bilinear form on the Grothendieck group of $C_2(\P)$. Here we use the same notation as for the Euler form on $\A$, since this should not cause confusion by the context.

\begin{lemma}
For any $M_\bullet,N_\bullet\in C_2(\P)$,
\begin{equation*}\lr{M_\bullet,N_\bullet}=\dim_k\Hom_{\A}(M_0,N_0)+\dim_k\Hom_{\A}(M_{-1},N_{-1})-\dim_k\Hom_{\A}(M_{-1},N_0).\end{equation*}
\end{lemma}
\begin{proof}
By Lemma \ref{resol}, we have a projective resolution (\ref{tfj}) of $M_{\bullet}$. Applying the functor $\Hom_{C_2(\P)}(-,N_\bullet)$ to the short exact sequence (\ref{tfj}), we obtain a long exact sequence
$$0\to\Hom(M_\bullet,N_\bullet)\to\Hom(C_{M_0}\oplus K_{M_{-1}},N_\bullet)\to\Hom(C_{M_{-1}},N_\bullet)\to\Ext^1(M_\bullet,N_\bullet)\to0,$$ since $C_{M_0}\oplus K_{M_{-1}}$ is projective.
Hence,
\begin{flalign*}
\lr{M_\bullet,N_\bullet}&=\dim_k\Hom_{C_2(\P)}(M_\bullet,N_\bullet)-\dim_k\Ext_{C_2(\P)}^1(M_\bullet,N_\bullet)\\
&=\dim_k\Hom_{C_2(\P)}(C_{M_0}\oplus K_{M_{-1}},N_{\bullet})-\dim_k\Hom_{C_2(\P)}(C_{M_{-1}},N_{\bullet}).
\end{flalign*}
Thus, by Lemma \ref{hom}, we complete the proof.
\end{proof}

\begin{lemma}\label{ext}
For any $M,N\in\A$ and $P\in\P$, there are the following
isomorphisms of vector spaces
\begin{flalign}
&\Ext^1_{C_2(\P)}(C_M,C_N)\cong\Ext^1_{\A}(M,N);\\
&\Ext^1_{C_2(\P)}(Z_P,C_M)\cong\Hom_{\A}(P,M)\label{e2}.
\end{flalign}
\end{lemma}
\begin{proof}
Firstly,
\begin{flalign*}\Ext^1_{C_2(\P)}(C_M,C_N)&\cong\Ext^1_{C^b(\P)}(C_M,C_N)\\
&\cong\Hom_{K^b(\P)}(C_M,C_N[1])\\
&\cong\Hom_{D^b(\A)}(M,N[1])\\
&\cong\Ext^1_{\A}(M,N).\end{flalign*}

Secondly,
\begin{flalign*}\Ext^1_{C_2(\P)}(Z_P,C_M)&\cong\Ext^1_{C^b(\P)}(Z_P,C_M)\\
&\cong\Hom_{K^b(\P)}(Z_P,C_M[1])\\
&\cong\Hom_{D^b(\A)}(P[1],M[1])\\
&\cong\Hom_{\A}(P,M).\end{flalign*}
\end{proof}
\begin{proposition}\label{kuozhang0}
Let $M_\bullet, N_\bullet\in C_2(\P)$. Then
\begin{equation}
\Ext^1_{C_2(\P)}(M_\bullet, N_\bullet)\cong\Hom_{\A}(H_{-1}(M_\bullet),H_{0}(N_\bullet))
\oplus\Ext^1_{\A}(H_{0}(M_\bullet),H_{0}(N_\bullet)).
\end{equation}
\end{proposition}
\begin{proof}
By Proposition \ref{fenjie}, we write $M_\bullet=K_{P}\oplus Z_Q\oplus C_M$ and $N_\bullet=K_{T}\oplus Z_W\oplus C_N$ for some $P,Q,T,W\in\P$ and $M,N\in\A$. Then
\begin{flalign*}
\Ext^1_{C_2(\P)}(M_\bullet, N_\bullet)&=\Ext^1_{C_2(\P)}(K_{P}\oplus Z_Q\oplus C_M,K_{T}\oplus Z_W\oplus C_N)\\
&\cong\Ext^1_{C_2(\P)}(Z_Q\oplus C_M,C_N)\\
&\cong\Hom_{\A}(Q,N)\oplus\Ext_{\A}^1(M,N)\\
&\cong\Hom_{\A}(H_{-1}(M_\bullet),H_{0}(N_\bullet))
\oplus\Ext^1_{\A}(H_{0}(M_\bullet),H_{0}(N_\bullet)).
\end{flalign*}
\end{proof}

\section{Hall algebras of morphisms}
Let $\H(C_2(\A))$ be the Hall algebra of the abelian category $C_2(\A)$ as defined in Definition \ref{Hall algebra of abelian category}. Let $\H(C_2(\P))$ be the submodule of $\H(C_2(\A))$ spanned by the isomorphism classes of objects in $C_2(\P)$. Since $C_2(\P)$ is closed under extensions, $\H(C_2(\P))$ is a subalgebra of the Hall algebra $\H(C_2(\A))$.
Define $\H_{\tw}(C_2(\P))$ to be the same module as $\H(C_2(\P))$, but with the twisted multiplication
$$[M_\bullet]\ast[N_\bullet]=q^{\lr{M_\bullet,N_\bullet}}[M_\bullet]\diamond[N_\bullet].$$

\begin{lemma}\label{KM}
For any $M_\bullet\in C_2(\P)$ and $P\in\P$, we have that in $\H_{\tw}(C_2(\P))$
\begin{equation}\label{jh}
[K_P]\ast[M_\bullet]=[M_\bullet]\ast[K_P]=[K_P\oplus M_\bullet].
\end{equation}
\end{lemma}
\begin{proof}
Since $K_P$ is projective-injective,
\begin{flalign*}
[K_P]\ast[M_\bullet]&=q^{\lr{K_P,M_\bullet}}\frac{1}{|\Hom_{C_2(\P)}(K_P,M_\bullet)|}[K_P\oplus M_\bullet]\\
&=[K_P\oplus M_\bullet].
\end{flalign*}
Similarly, $[M_\bullet]\ast[K_P]=[K_P\oplus M_\bullet].$ Thus, we complete the proof.
\end{proof}
\begin{lemma}\label{ybasis}
For any $P,Q\in\P$ and $M\in\A$, we have that in $\H_{\tw}(C_2(\P))$
$$[K_P\oplus Z_Q\oplus C_M]=[K_P]\ast[C_M]\ast[Z_Q].$$
\end{lemma}
\begin{proof}
Since $\Ext_{C_2(\P)}^1(C_M,Z_Q)=0$, we obtain that
\begin{flalign*}[C_M]\ast[Z_Q]&=q^{\lr{C_M,Z_Q}}\frac{1}{|\Hom_{C_2(\P)}(C_M,Z_Q)|}[C_M\oplus Z_Q]\\&=[C_M\oplus Z_Q].\end{flalign*}
Thus, by Lemma \ref{KM}, $$[K_P\oplus Z_Q\oplus C_M]=[K_P]\ast[C_M\oplus Z_Q]=[K_P]\ast[C_M]\ast[Z_Q].$$
\end{proof}

Define the {\em localized Hall algebra} $\M\H(\A)$ to be the localization of $\H_{\tw}(C_2(\P))$ with respect to all elements $[K_P]$. In symbols,
$$\M\H(\A):=\H_{\tw}(C_2(\P))[[K_P]^{-1}: P\in\P].$$

For each $\alpha\in K(\A)$, by writing $\alpha=\hat{P}-\hat{Q}$ for some objects $P,Q\in\P$, we define
$$K_{\alpha}:=[K_P]\ast[K_Q]^{-1}.$$ Then for any $\alpha,\beta\in K(\A)$ and $M_\bullet\in C_2(\P)$, \begin{flalign}&K_\alpha\ast K_\beta=K_{\alpha+\beta}=K_\beta\ast K_\alpha;\label{KK}\\
&K_\alpha\ast[M_\bullet]=[M_\bullet]\ast K_\alpha.\label{KX}\end{flalign}

For each $M\in\A$, we define
$$\X_M:=K_{-\hat{P}_M}\ast[C_M]\in\M\H(\A).$$
Suppose we take a different, not necessarily minimal, projective resolution (\ref{mpr}), and consider the corresponding object, denoted by $C'_M$, in $C_2(\P)$. By Lemma \ref{jx}, $[C'_M]=[K_R\oplus C_M]$ for some $R\in\P$. Thus
\begin{flalign*}K_{-(\hat{P}_M+\hat{R})}\ast[C'_M]&=K_{-(\hat{P}_M+\hat{R})}\ast K_{\hat{R}}\ast [C_M]\\&=
K_{-\hat{P}_M}\ast [C_M]\\&=\X_M.
\end{flalign*}
In other word, $\X_M$ does not depend on the minimality of projective resolutions of $M$.

Given $M\in\A$ and $P\in\P$, we define
\begin{flalign*}\X_{M\oplus P[1]}:=\X_M\ast[Z_P]=K_{-\hat{P}_M}\ast[C_M]\ast[Z_P]=K_{-\hat{P}_M}\ast[C_M\oplus Z_P].
\end{flalign*}
In particular, \begin{equation}\label{XMP}\X_{P[1]}=[Z_P]~~\text{and}~~\X_{M\oplus P[1]}=\X_M\ast \X_{P[1]}.\end{equation}
\begin{proposition}\label{basis}
The algebra $\M\H(\A)$ has a basis consisting of elements
\begin{equation*}K_{\alpha}\ast \X_M\ast \X_{P[1]},\end{equation*}where $\alpha\in K(\A)$, $M\in\A$ and $P\in\P$.
\end{proposition}
\begin{proof}
It can be easily proved by Proposition \ref{fenjie} and Lemma \ref{ybasis}.
\end{proof}
\begin{lemma}\label{XP}
For any $P,Q\in\P$, we have that in $\M\H(\A)$
\begin{equation}\X_{P[1]}\ast \X_{Q[1]}=\X_{Q[1]}\ast \X_{P[1]}=\X_{(P\oplus Q)[1]}.\end{equation}
\end{lemma}
\begin{proof}
By definition,
\begin{flalign*}
\X_{P[1]}\ast \X_{Q[1]}&=q^{\lr{Z_P,Z_Q}}\frac{1}{|\Hom_{C_2(\P)}(Z_P,Z_Q)|}\X_{(P\oplus Q)[1]}\\
&=\X_{(P\oplus Q)[1]}.
\end{flalign*}
\end{proof}
\begin{theorem}\label{main1}
There exists an embedding of algebras
$$\xymatrix{\Psi:\H_{\tw}(\A)\ar@{^{(}->}[r]&\M\H(\A),} \xymatrix{[M]\ar@{|->}[r]&\X_M.}$$
\end{theorem}
\begin{proof}
By Proposition \ref{basis}, clearly, $\Psi$ is injective. We only need to prove that $\Psi$ is a homomorphism of algebras.

By definition,
\begin{flalign*}
\X_M\ast \X_N&=K_{-\hat{P}_M}\ast[C_M]\ast K_{-\hat{P}_N}\ast[C_N]\\
&=K_{-(\hat{P}_M+\hat{P}_N)}\ast[C_M]\ast[C_N].
\end{flalign*}
By Lemma \ref{ext}, $$\Ext^1_{C_2(\P)}(C_M,C_N)\cong\Ext^1_{\A}(M,N).$$ Moreover, by the Horse-Shoe Lemma (cf. \cite{Weibel}), any extension of $C_M$ by $C_N$ is the object $C'_L$ defined by the corresponding extension $L$ of $M$ by $N$. Hence,
$$|\Ext^1_{C_2(\P)}(C_M,C_N)_{C'_L}|=|\Ext^1_{\A}(M,N)_L|.$$

Thus,
\begin{flalign*}
\X_M\ast \X_N&=q^{\lr{C_M,C_N}}\sum\limits_{[L]}\frac{|\Ext^1_{C_2(\P)}(C_M,C_N)_{C'_L}|}{|\Hom_{C_2(\P)}(C_M,C_N)|} K_{-(\hat{P}_M+\hat{P}_N)}\ast[C'_L]\\
&=\frac{|\Hom_{C_2(\P)}(C_M,C_N)|}{|\Ext^1_{C_2(\P)}(C_M,C_N)|}\sum\limits_{[L]}\frac{|\Ext^1_{\A}(M,N)_L|}{|\Hom_{C_2(\P)}(C_M,C_N)|} \X_L\\
&=\sum\limits_{[L]}\frac{|\Ext^1_{\A}(M,N)_L|}{|\Ext^1_{C_2(\P)}(C_M,C_N)|} \X_L\\
&=\sum\limits_{[L]}\frac{|\Ext^1_{\A}(M,N)_L|}{|\Ext^1_{\A}(M,N)|} \X_L\\
&=\frac{|\Hom_{\A}(M,N)|}{|\Ext^1_{\A}(M,N)|}\sum\limits_{[L]}\frac{|\Ext^1_{\A}(M,N)_L|}{|\Hom_{\A}(M,N)|} \X_L\\
&=q^{\lr{M,N}}\sum\limits_{[L]}\frac{|\Ext^1_{\A}(M,N)_L|}{|\Hom_{\A}(M,N|} \X_L.
\end{flalign*}
That is, $\Psi([M])\ast\Psi([N])=\Psi([M]\ast[N])$, and we complete the proof.
\end{proof}

\section{A multiplication formula in $\M\H(\A)$}
Given $M\in\A$ and $P\in\P$, consider an extension of $C_M$ by $Z_P$
\begin{equation}\label{eta}\eta: 0\longrightarrow C_M\longrightarrow X_{\bullet}\longrightarrow Z_P\longrightarrow0.\end{equation}
It induces a long exact sequence in homology
$$H_{-1}(C_M)\longrightarrow H_{-1}(X_\bullet)\longrightarrow H_{-1}(Z_P)\longrightarrow H_0(C_M)\longrightarrow H_0(X_{\bullet})\longrightarrow H_0(Z_P).$$
Since $H_{-1}(C_M)=H_0(Z_P)=0$, $H_{-1}(Z_P)\cong P$ and $H_0(C_M)\cong M$, we obtain the exact sequence
$$0\longrightarrow H_{-1}(X_\bullet)\longrightarrow P\longrightarrow M\longrightarrow H_0(X_{\bullet})\longrightarrow0.$$
Writing $X_\bullet=K_{T}\oplus Z_{Q}\oplus C_B$ for some $T,Q\in\P$ and $B\in\A$, we get the following exact sequence
$$\xymatrix{0\ar[r]& Q\ar[r]& P\ar[r]^-{\delta}& M\ar[r]& B\ar[r]& 0,}$$
where $\delta$ is determined by the equivalence class of $\eta$ via the isomorphism
\begin{equation}\label{tg}
\Ext^1_{C_2(\P)}(Z_P,C_M)\cong\Hom_{\A}(P,M).\end{equation}
By the short exact sequence (\ref{eta}), we obtain that
\begin{equation}P_B\oplus T\cong P_M~~\text{and}~~\Omega_B\oplus Q\oplus T\cong P\oplus\Omega_M.\end{equation}
It is easy to see that the isomorphism (\ref{tg}) induces that
\begin{equation}\label{kuozhang}\Ext^1_{C_2(\P)}(Z_P,C_M)_{K_T\oplus Z_Q\oplus C_B}\cong {}_Q\Hom_{\A}(P,M)_B,\end{equation}
where ${}_Q\Hom_{\A}(P,M)_B:=\{f:P\to M~|~\Ker(f)\cong Q~~\text{and}~~\Coker(f)\cong B\}$. By \cite{Vanden,ZHC},
\begin{equation}\label{xjs}|{}_Q\Hom_{\A}(P,M)_B|=\sum\limits_{[L]}a_LF^P_{LQ}F^M_{BL}.\end{equation}

\begin{theorem}\label{main2}
Given $M\in\A$ and $P\in\P$, we have that in $\M\H(\A)$
\begin{flalign*}\X_{P[1]}\ast \X_M&=q^{-\lr{P,M}}\sum\limits_{[B],[Q]}|{}_Q\Hom_{\A}(P,M)_B| \X_B\ast \X_{Q[1]}\\
&=q^{-\lr{P,M}}\sum\limits_{[B],[Q]}|{}_Q\Hom_{\A}(P,M)_B| \X_{B\oplus Q[1]}.\end{flalign*}
\end{theorem}
\begin{proof}
By definition,
\begin{flalign*}
\X_{P[1]}\ast \X_M&=[Z_P]\ast K_{-\hat{P}_M}\ast[C_M]\\
&=K_{-\hat{P}_M}\ast[Z_P]\ast[C_M]\\
&=q^{\lr{Z_P,C_M}}\sum\limits_{[B],[T],[Q]}\frac{|\Ext^1_{C_2(\P)}(Z_P,C_M)_{K_T\oplus Z_Q\oplus C_B}|}{|\Hom_{C_2(\P)}(Z_P,C_M)|} K_{-\hat{P}_M}\ast[K_T\oplus Z_Q\oplus C_B].
\end{flalign*}
Since $\Hom_{C_2(\P)}(Z_P,C_M)=0$ and then $\lr{Z_P,C_M}=-\lr{P,M}$, by (\ref{kuozhang}), we obtain that
\begin{flalign*}
\X_{P[1]}\ast \X_M
&=q^{-\lr{P,M}}\sum\limits_{[B],[T],[Q]}|{}_Q\Hom_{\A}(P,M)_B| K_{-\hat{P}_M}\ast[K_T\oplus Z_Q\oplus C_B]\\
&=q^{-\lr{P,M}}\sum\limits_{[B],[T],[Q]}|{}_Q\Hom_{\A}(P,M)_B| K_{-\hat{P}_M}\ast K_{\hat{T}}\ast[C_B]\ast[Z_Q]\\
&=q^{-\lr{P,M}}\sum\limits_{[B],[Q]}|{}_Q\Hom_{\A}(P,M)_B| K_{-\hat{P}_B}\ast[C_B]\ast[Z_Q]\\
&=q^{-\lr{P,M}}\sum\limits_{[B],[Q]}|{}_Q\Hom_{\A}(P,M)_B| \X_B\ast \X_{Q[1]}\\
&=q^{-\lr{P,M}}\sum\limits_{[B],[Q]}|{}_Q\Hom_{\A}(P,M)_B| \X_{B\oplus Q[1]}.
\end{flalign*}
\end{proof}

\begin{corollary}
Given $M\in\A$ and $P\in\P$, if $\Hom_{\A}(P,M)=0$, then in $\M\H(\A)$
$$\X_{P[1]}\ast \X_M=\X_M\ast \X_{P[1]}.$$
\end{corollary}

\section{Generators in $\M\H(\A)$}
Let $\Ind(\A)$ be a complete set of indecomposable objects in $\A$,
and let $\Ind(\P)\subset\Ind(\A)$ be a subset consisting of indecomposable projective objects. Define $K_{\geq0}(\A)$ to be the positive cone in $K(\A)$, which consists of classes of objects in $\A$, rather than formal differences of such. We assume that for each nonzero object $M\in\A$, $\hat{M}\neq0$. This condition is equivalent to the statement that the rule
$$\alpha\leq\beta~~\Leftrightarrow~~\beta-\alpha\in K_{\geq0}(\A)$$ defines a partial order on $K(\A)$. It holds, for example, if $\A$ has finite length.

\begin{proposition}\label{generators}
The set $\{K_{{P}}^{\pm},\X_M,\X_{P[1]}~|~M\in\Ind(\A)~~\text{and}~~P\in\Ind(\P)\}$ is a generating set of $\M\H(\A)$.
\end{proposition}
\begin{proof}
For any short exact sequence in $C_2(\P)$
$$0\longrightarrow N_\bullet\longrightarrow L_\bullet \longrightarrow M_\bullet\longrightarrow0,$$
it induces a long exact sequence in homology
{\footnotesize\begin{equation}\label{f}\xymatrix{0\to H_{-1}(N_\bullet)\ar[r]& H_{-1}(L_\bullet)\ar[r]& H_{-1}(M_\bullet)\ar[r]^-f& H_{0}(N_\bullet)\ar[r]& H_{0}(L_\bullet)\ar[r]& H_{0}(M_\bullet)\to0.}\end{equation}}
Setting $Y:=\im(f)$, we obtain two exact sequences
$$\xymatrix{0\to H_{-1}(N_\bullet)\ar[r]& H_{-1}(L_\bullet)\ar[r]& H_{-1}(M_\bullet)\ar[r]^-f&Y\to0}$$
and
$$\xymatrix{0\to Y\ar[r]&H_{0}(N_\bullet)\ar[r]& H_{0}(L_\bullet)\ar[r]& H_{0}(M_\bullet)\to0.}$$
Thus, \begin{equation}\begin{split}\label{filter}&\Dim H_{-1}(L_\bullet)+\Dim H_{0}(L_\bullet)=\\
&\Dim H_{-1}(M_\bullet)+\Dim H_{0}(M_\bullet)+\Dim H_{-1}(N_\bullet)+\Dim H_{0}(N_\bullet)-2\Dim Y.\end{split}\end{equation}

For any object $X_\bullet\in C_2(\P)$, define $\Dim X_\bullet:=\Dim H_{-1}(X_\bullet)+\Dim H_{0}(X_\bullet)$. Then (\ref{filter}) can be rewritten as \begin{equation}\label{filter2}\Dim L_\bullet=\Dim M_\bullet+\Dim N_\bullet-2\Dim Y.\end{equation} For each $\alpha\in K_{\geq0}(\A)$, let $\H_{tw}^{\leq\alpha}(C_2(\P))$ be the submodule of $\H_{tw}(C_2(\P))$ spanned by all $[X_\bullet]$ with $\Dim X_\bullet\leq\alpha$. By (\ref{filter2}),
$$\H_{tw}^{\leq\alpha}(C_2(\P))\ast \H_{tw}^{\leq\beta}(C_2(\P))\subseteq\H_{tw}^{\leq\alpha+\beta}(C_2(\P)).$$
That is, $\H_{tw}(C_2(\P))$ is a $K_{\geq0}(\A)$-filtered algebra.

By Proposition \ref{kuozhang0}, $$\Ext^1_{C_2(\P)}(M_\bullet, N_\bullet)\cong\Hom_{\A}(H_{-1}(M_\bullet),H_{0}(N_\bullet))
\oplus\Ext^1_{\A}(H_{0}(M_\bullet),H_{0}(N_\bullet)).$$
If $f$ in (\ref{f}) is nonzero, then $\Dim L_\bullet<\Dim M_\bullet+\Dim N_\bullet$;
otherwise, $\Dim L_\bullet=\Dim M_\bullet+\Dim N_\bullet$ and we have the short exact sequence $$\xi: 0\longrightarrow H_0(N_\bullet)\longrightarrow H_0(L_\bullet)\longrightarrow H_0(M_\bullet)\longrightarrow0.$$ Furthermore, $$\dim_k\End_{\A}(H_0(L_\bullet))\leq\dim_k\End_{\A}(H_0(M_\bullet)\oplus H_0(N_\bullet)),$$ and the equality holds if and only if $\xi$ is splitting.

Let us give an order on $K_{\geq0}(\A)\times\mathbb{N}$:
$$(\alpha_1,d_1)\leq(\alpha_2,d_2)\Longleftrightarrow``\alpha_1<\alpha_2"~ \text{or}~``\alpha_1=\alpha_2,d_1\leq d_2".$$

For any $X_\bullet\in C_2(\mathscr{P})$, set $\deg X_\bullet:=(\Dim X_\bullet,\dim_k\End_{\A}(H_0(X_\bullet)))$.
Then for any $M_\bullet,N_\bullet\in C_2(\mathscr{P})$,
$$[M_\bullet]\ast[N_\bullet]=a_{M_\bullet\oplus N_\bullet}[M_\bullet\oplus N_\bullet]+\sum\limits_{[L_\bullet]:\deg L_\bullet<\deg M_\bullet\oplus N_\bullet}a_{L_\bullet}[L_\bullet],$$
where $a_{M_\bullet\oplus N_\bullet}$ and all $a_{L_\bullet}$ belong to $\ZZ[{q}^{\pm\frac{1}{2}}]$.

By induction on the number of indecomposable direct summands, we prove that the elements in $\{[K_P],[C_M],[Z_P]~|~M\in\Ind(\A)~~\text{and}~~P\in\Ind(\P)\}$ generate the Hall algebra $\H_{\tw}(C_2(\P))$, and then we easily complete the proof.
\end{proof}

For use below we write down the defining relations of $\mathcal{MH}(\A)$ in the following
\begin{theorem}\label{ydygx}
The algebra $\mathcal{MH}(\A)$ is generated by all $K_{\alpha}$ and $\mathbb{X}_{M\oplus P[1]}$ $($with $\alpha\in K(\A)$, $M\in\A$ and $P\in\P$$)$,
which are subject to the following relations
\begin{equation}\label{h1}
\begin{split}K_{\alpha}\ast K_{\beta}=K_{\alpha+\beta}=K_{\beta}\ast K_{\alpha};\end{split}\end{equation}
\begin{equation}\label{h2}
\begin{split}K_{\alpha}\ast\mathbb{X}_{M\oplus P[1]}
    =\mathbb{X}_{M\oplus P[1]}\ast K_{\alpha};\end{split}\end{equation}
\begin{equation}\label{h3}
\begin{split}
\mathbb{X}_{P[1]}\ast \mathbb{X}_{Q[1]}=
\mathbb{X}_{(P\oplus Q)[1]}
=\mathbb{X}_{Q[1]}\ast\mathbb{X}_{P[1]};
\end{split}\end{equation}
\begin{equation}\label{h4}\mathbb{X}_{M}\ast\mathbb{X}_{N}=q^{\lr{M,N}}\sum_{[L]}\frac{|\mathrm{Ext}_{\A}^{1}(M,N)_{L}|}{|\mathrm{Hom}_{\A}(M,N)|}\mathbb{X}_L;\end{equation}
\begin{equation}\label{h5}\mathbb{X}_{M}\ast\mathbb{X}_{P[1]}=\mathbb{X}_{M\oplus
P[1]};\end{equation}
\begin{equation}\label{h6}\mathbb{X}_{P[1]}\ast\mathbb{X}_{M}=q^{-\lr{P,M}}
\sum\limits_{[B],[Q]}|{}_Q\Hom_{\A}(P,M)_B|\mathbb{X}_{B\oplus Q[1]};\end{equation}
for any $\alpha,\beta\in K(\A)$, $M,N\in\A$ and $P,Q\in\P$.
\end{theorem}
\begin{proof}
By Proposition \ref{basis},
$\mathcal{MH}(\A)$ is spanned by all elements $K_{\alpha}\ast\mathbb{X}_{M\oplus P[1]}=K_{\alpha}\ast\mathbb{X}_{M}\ast\mathbb{X}_{P[1]}$. That is, the multiplication map induces a triangular decomposition of $\mathcal{MH}(\A)$ as a module.
Using the triangular decomposition of $\mathcal{MH}(\A)$, we can easily prove that the relations $(\ref{h1}$-$\ref{h6})$ are the defining relations, which have been obtained in the previous sections.
\end{proof}

\section{Derived Hall algebras}
The derived Hall algebra of the bounded derived category $D^b(\A)$ of $\A$ was introduced in \cite{Toen2006} (see also \cite{XiaoXu}). By definition, the (Drinfeld dual) {\em derived Hall algebra} $\mathcal {D}\mathcal {H}(\A)$ is the free $\ZZ[{q}^{\pm\frac{1}{2}}]$-module with the basis $\{u_{X_\bullet}~|~X_\bullet\in \Iso(D^b(\A))\}$ and the multiplication defined by
\begin{equation}
u_{X_\bullet}\diamond u_{Y_\bullet}=\sum\limits_{[{Z_\bullet}]}\frac{|\Ext^1_{D^b(\A)}({X_\bullet},{Y_\bullet})_{Z_\bullet}|}{\prod\limits_{i\geq0}|\Hom_{D^b(\A)}({X_\bullet}[i],{Y_\bullet})|^{(-1)^i}} u_{Z_\bullet},
\end{equation}
where $\Ext^1_{D^b(\A)}({X_\bullet},{Y_\bullet})_{Z_\bullet}$ is defined to be $\Hom_{D^b(\A)}({X_\bullet},{Y_\bullet}[1])_{{Z_\bullet}[1]}$, which denotes the subset of $\Hom_{D^b(\A)}({X_\bullet},{Y_\bullet}[1])$ consisting of morphisms $f:{X_\bullet}\rightarrow {Y_\bullet}[1]$ whose cone is isomorphic to ${Z_\bullet}[1]$.

Let us  reformulate \cite[Proposition 7.1]{Toen2006} in the form of the Drinfeld dual derived Hall algebras
as the following
\begin{proposition}{\rm(\cite{Toen2006})}
$\mathcal {D}\mathcal {H}(\A)$ is an associative unital algebra generated by the elements in $\{u_{M[i]}~|~M\in\Iso(\A),~i\in \mathbb{Z}\}$ and the following relations
\begin{flalign}
&u_{M[i]}\diamond u_{N[i]}=\sum\limits_{[L]}{\frac{{|\Ext_\mathcal{A}^1{{(M,N)}_L}|}}{{|\Hom_\mathcal{A}(M,N)|}}} u_{L[i]};\\
&u_{M[i+1]}\diamond u_{N[i]}=\sum\limits_{[X],[Y]}q^{-\lr{{Y},{X}}}|{}_X\Hom_{\A}(M,N)_Y| u_{Y[i]}\diamond u_{X[i+1]};\\
&u_{M[i]}\diamond u_{N[j]}=q^{(-1)^{i-j}\lr{N,M}} u_{N[j]}\diamond u_{M[i]}, \quad i-j>1.
\end{flalign}
\end{proposition}

For any ${X_\bullet},{Y_\bullet}\in D^b(\A)$, define
\begin{equation*}
\lr{{X_\bullet},{Y_\bullet}}:=\sum\limits_{i\in\mathbb{Z}}(-1)^i\dim_k\Hom_{D^b(\A)}({X_\bullet},{Y_\bullet}[i]),
\end{equation*}
it also descends to give a bilinear form on the Grothendieck group of $D^b(\A)$. Moreover, it coincides with the Euler form of $K(\A)$ over the objects in $\A$. In particular, for any $M,N\in\A$ and $i,j\in\mathbb{Z}$, we have that $\lr{M[i],N[j]}=(-1)^{i-j}\lr{M,N}$.

Let us twist the multiplication in $\mathcal {D}\mathcal {H}(\A)$ as follows (cf. \cite{SX}):
\begin{equation}u_{X_\bullet}\ast u_{Y_\bullet}=q^{\lr{{X_\bullet},{Y_\bullet}}} u_{X_\bullet}\diamond u_{Y_\bullet}\end{equation}
for any ${X_\bullet},{Y_\bullet}\in D^b(\A)$.
The \emph{twisted derived Hall algebra} $\mathcal {D}\mathcal {H}_{q}(\A)$ is the same module as $\mathcal {D}\mathcal {H}(\A)$, but with the twisted multiplication. Then we have the following
\begin{proposition}\label{twistderived}
$\mathcal {D}\mathcal {H}_q(\A)$ is an associative unital algebra generated by the elements in $\{u_{M[i]}~|~M\in\Iso(\A),~i\in \mathbb{Z}\}$ and the following relations
\begin{flalign}
&u_{M[i]}\ast u_{N[i]}=q^{\lr{M,N}}\sum\limits_{[L]}{\frac{{|\Ext_\mathcal{A}^1{{(M,N)}_L}|}}{{|\Hom_\mathcal{A}(M,N)|}}}u_{L[i]};\\
&u_{M[i+1]}\ast u_{N[i]}=q^{-\lr{M,N}}\sum\limits_{[X],[Y]}|{}_X\Hom_{\A}(M,N)_Y| u_{Y[i]}\ast u_{X[i+1]};\\
&u_{M[i]}\ast u_{N[j]}=q^{(-1)^{i-j}\lr{M,N}} u_{N[j]}\ast u_{M[i]}, \quad i-j>1.
\end{flalign}
\end{proposition}
\begin{remark}
For any $M,N\in\A$ and $i\in\mathbb{Z}$, we have that
$$u_{M[i]}\ast u_{N[i+1]}=u_{M[i]\oplus N[i+1]}.$$
In fact, since $\A$ is hereditary, we obtain that \begin{flalign*}\Hom_{D^b(\A)}(M[i],N[i+1][1])\cong\Ext^2_{\A}(M,N)=0.\end{flalign*}
Thus, \begin{flalign*}u_{M[i]}\diamond u_{N[i+1]}&=\frac{1}{\prod\limits_{j\geq0}|\Hom_{D^b(\A)}(M[i+j],N[i+1])|^{(-1)^j}}u_{M[i]\oplus N[i+1]}\\&=q^{\lr{M,N}} u_{M[i]\oplus N[i+1]}.\end{flalign*}
Hence, $$u_{M[i]}\ast u_{N[i+1]}=q^{-\lr{M,N}} u_{M[i]}\diamond u_{N[i+1]}=u_{M[i]\oplus N[i+1]}.$$
\end{remark}
In order to compare with the subsequent cluster multiplication formulas, we give the following
\begin{corollary}
Given objects $M\in\A$, $P\in\P$ and an injective object $I\in\A$, we have that
\vspace{0.2cm}
\begin{itemize}
\item[(1)] $u_{M}\ast u_{I[-1]}=q^{-\lr{M,I}}\sum\limits_{[B],[I']}|{}_B\Hom_{\A}(M,I)_{I'}| u_{B\oplus I'[-1]};$
\vspace{0.4cm}
\item[(2)] $u_{P[1]}\ast u_M=q^{-\lr{P,M}}\sum\limits_{[B],[Q]}|{}_Q\Hom_{\A}(P,M)_B| u_{B\oplus Q[1]}$.
\end{itemize}
\end{corollary}

Let $\mathbb{T}(\A)$ be the group algebra of the Grothendieck group $K(\A)$ over $\ZZ[{q}^{\pm\frac{1}{2}}]$. For each $\alpha\in K(\A)$, we denote by $\T_\alpha$ the element in $\mathbb{T}(\A)$ corresponding to $\alpha$, thus $\T_\alpha\ast\T_\beta=\T_{\alpha+\beta}$. We equip the module $\mathcal {D}\mathcal {H}_q(\A)\otimes_{\ZZ[{q}^{\pm\frac{1}{2}}]}\mathbb{T}(\A)$ with the structure of an algebra (containing $\mathcal {D}\mathcal {H}_q(\A)$ and $\mathbb{T}(\A)$ as subalgebras) by imposing the relations
$\T_\alpha\ast x=x\ast\T_\alpha$ for any $\alpha\in K(\A)$ and $x\in \mathcal {D}\mathcal {H}_q(\A)$, and denote this algebra by $\mathcal {D}\mathcal {H}^{\e}_q(\A)$. We remark that $\mathcal {D}\mathcal {H}^{\e}_q(\A)$ is just the tensor algebra of $\mathcal {D}\mathcal {H}_q(\A)$ and $\mathbb{T}(\A)$.
\begin{theorem}\label{mor_derived}There exists an embedding of algebras $$\xymatrix{\Phi:\M\H(\A)\ar@{^{(}->}[r]&
\mathcal{D}\mathcal{H}_q^{\e}(\A)}$$
defined on generators by
$$K_\alpha\mapsto\T_\alpha,\quad\X_M\mapsto u_M\quad\text{and}\quad \X_{P[1]}\mapsto u_{P[1]}$$
for $\alpha\in K(\A)$, $M\in\A$ and $P\in\P$.
\end{theorem}
\begin{proof}
By Proposition \ref{twistderived} and the definition of $\mathcal{D}\mathcal{H}_q^{\e}(\A)$, we know that
all the relations in Theorem \ref{ydygx} are preserved under $\Phi$, so $\Phi$ is a homomorphism of algebras.
The injectivity of $\Phi$ follows from the fact that $\Phi$ sends the basis
$$\{K_\alpha\ast\X_{M\oplus P[1]}~|~\alpha\in K(\A),M\in\A,P\in\P\}$$
of $\M\H(\A)$ to a linearly independent set in $\mathcal{D}\mathcal{H}_q^{\e}(\A)$.
\end{proof}

\section{Cluster multiplication formulas for acyclic quivers}
In this section, we recall the definitions of the quantum cluster algebra and quantum Caldero--Chapoton map, and give the multiplication formulas of quantum cluster characters.

\subsection{Quantum cluster algebras}
Let $\Lambda$ be an $m\times m$ skew-symmetric integral matrix, and denote by $\{{\bf e}_1,\cdots,{\bf e}_m\}$ the standard basis of $\mathbb{Z}^{m}$. Let $\mathfrak{q}$ be a formal variable
and $\ZZ[\mathfrak{q}^{\pm\frac{1}{2}}]$ be the ring of integral Laurent polynomials.
Define the  quantum torus  associated to the pair
$(\mathbb{Z}^{m},\Lambda)$ to be the $\ZZ[\mathfrak{q}^{\pm\frac{1}{2}}]$-algebra $\mathcal{T}_{\mathfrak{q}}$ with
a distinguished basis $\{X^{\bf e}: {\bf e}\in \mathbb{Z}^{m}\}$ and the
multiplication given by
\[X^{\bf e}X^{\bf f}=\mathfrak{q}^{\Lambda({\bf e},{\bf f})/2}X^{{\bf e+f}},\]
where we still denote by $\Lambda$ the skew-symmetric bilinear form on $\mathbb{Z}^{m}$ associated to the skew-symmetric matrix $\Lambda$. It is well-known that $\mathcal{T}_{\mathfrak{q}}$ is an Ore domain, and thus is contained in its
skew-field of fractions $\mathcal{F}_{\mathfrak{q}}$.

Let $\tilde{B}=(b_{ij})$ be an $m\times n$ integral matrix with $n\le m$.  We call the pair
$(\Lambda, \tilde{B})$ {\em compatible} if $\tilde{B}^{tr}\Lambda=(D|0)$ for some
$D=diag(d_1,\cdots,d_n)$, where each $d_i$ is a positive integer and  $\tilde{B}^{tr}$ denotes the transpose of  $\tilde{B}$. An {\em initial  quantum seed} for $\mathcal{F}_{\mathfrak{q}}$ is a triple
$(\Lambda, \tilde{B}, X)$  consisting of a compatible pair $(\Lambda,\tilde{B})$ and the set  $X=\{X_1,\cdots,X_m\}$, where each $X_i$ denotes $X^{{\bf e_i}}$. For any $1\leq k\leq n$, we  define the mutation of $(\Lambda, \tilde{B}, X)$ in direction $k$ to obtain the new quantum seed $(\Lambda',\tilde{B}',X')$ as follows:

(1)\ $\Lambda'=E^{tr}\Lambda E$, where the
$m\times m$ matrix $E=(e_{ij})$ is given by
\[e_{ij}=\begin{cases}
\delta_{ij} & \text{if $j\ne k$;}\\
-1 & \text{if $i=j=k$;}\\
max(0,-b_{ik}) & \text{if $i\ne j = k$.}
\end{cases}
\]

(2)\ $\tilde{B}'=(b'_{ij})$ is given by
\[b'_{ij}=\begin{cases}
-b_{ij} & \text{if $i=k$ or $j=k$;}\\
b_{ij}+\frac{|b_{ik}|b_{kj}+b_{ik}|b_{kj}|}{2} & \text{otherwise.}
\end{cases}
\]

(3)\  $X'=\{X'_1,\cdots,X'_m\}$ is given by
\begin{align}
X_k'&=X^{\sum_{1\leq i\leq m}[b_{ik}]_{+} {\bf e}_i -{\bf e}_k}+X^{\sum_{1\leq i\leq m}[-b_{ik}]_{+} {\bf e}_i -{\bf e}_k},\nonumber\\
X_i'&=X_i ,\quad 1\leq i\leq m,\quad i\neq k,\nonumber
\end{align}
where for each integer $a$ we set $[a]_{+}:=max\{0,a\}$.

Two quantum seeds  $(\Lambda, \tilde{B}, X)$ and  $(\Lambda', \tilde{B}', X')$ are
called {\em mutation-equivalent}, denoted by $(\Lambda, \tilde{B}, X)\sim(\Lambda', \tilde{B}', X')$, if they can be obtained from each other
by a sequence of mutations. Let $\mathcal{C}=\{X'_i~|~ (\Lambda', \tilde{B}', X')\sim(\Lambda, \tilde{B}, X),1\leq i\leq n\}$, and
the elements in $\mathcal{C}$ are called the {\em quantum cluster
variables}. Let $\mathbb{P}=\{X_i~|~n+1\leq i\leq m\}$, and the
elements in $\mathbb{P}$ are called the {\em coefficients}. Denote by
$\ZZ\mathbb{P}$ the ring of  Laurent polynomials in the elements of $\mathbb{P}$ with coefficients in $\ZZ[\mathfrak{q}^{\pm\frac{1}{2}}]$. Then the
{\em quantum cluster algebra}
$\mathcal{A}_{\mathfrak{q}}(\Lambda,\tilde{B})$ is defined to be the
$\ZZ\mathbb{P}$-subalgebra of $\mathcal{F}_{\mathfrak{q}}$ generated by all
quantum cluster variables.

\subsection{Quantum Caldero--Chapoton map and cluster multiplication formulas }\label{quantum_cluster_char}
Fix a positive integer $n$ and let $Q=\{Q_0, Q_1, s, t\}$ be an acyclic quiver with the vertex set $Q_0=
\{1, 2,\cdots, n\}$ and arrow set $Q_1$. We denote by $s(\rho)$ and $t(\rho)$ the source and
target of an arrow $\rho\in Q_1$, respectively. Take $\A=\A_{Q}$ to be the category of finite dimensional $kQ$-modules. Then the Grothendieck group $K(\A)$ is isomorphic to $\mathbb{Z}Q_0\cong\mathbb{Z}^n$. There is a bilinear form $\lr{-,-}:\mathbb{Z}^n\times\mathbb{Z}^n\to\mathbb{Z}$ defined by
$$\lr{{\bf x},{\bf y}}=\sum\limits_{i\in Q_0}x_iy_i-\sum\limits_{\rho\in Q_1}x_{s(\rho)}y_{t(\rho)}$$
for ${\bf x}=(x_i),{\bf y}=(y_i)\in\mathbb{Z}^n$,
which is called the {\em Euler form} of $Q$. It is well known that this form coincides with the (homological) Euler form defined in (\ref{Euler form}).

Let $B(Q)$ and $R(Q)$ be the $n\times n$ matrixes with the $i$-th row and $j$-th column element given respectively by
$$b_{ij}=|\{\mathrm{arrows}\, i\longrightarrow
j\}|-|\{\mathrm{arrows}\, j\longrightarrow i\}|$$
and
$$r_{ij}=|\{\mathrm{arrows}\, j\longrightarrow i\}|.$$

By definition, it is easy to see that for any ${\bf x}, {\bf y}\in\mathbb{Z}^n$
$$\lr{{\bf x},{\bf y}}={\bf x}^{tr}(I_n-R(Q)^{tr}){\bf y},$$
where $I_n$ is the $n\times n$ identity matrix. Moreover, for each projective $kQ$-module $P$,
$$(I_n-R(Q))\Dim P=\Dim (P/\rad P).$$

For each fixed integer $m\geq n$, we choose a quiver $\widetilde{Q}$ with the vertex set $\{1,2,\cdots,m\}$ such that $Q$ is a full subquiver of $\widetilde{Q}$.
Let $\widetilde{I}$, $\widetilde{B}$, $\widetilde{R}$ and $\widetilde{R}'$ be the
left $m\times n$ submatrixes of the matrixes $I_{m}$, $B(\widetilde{Q})$, $R(\widetilde{Q})$ and $R(\widetilde{Q})^{tr}$, respectively. For concision, we write $B$ and $R$ for $B(Q)$ and $R(Q)$, respectively.
Note that $\widetilde{B}=\widetilde{R}'-\widetilde{R}$ and $B=R^{tr}-R$.
We always assume that there exists a skew-symmetric $m\times m$ integral
matrix $\Lambda$ such that
\begin{align}\label{eq:simply_laced_compatible}
\Lambda(-\widetilde{B})={I_n\choose0}.\end{align}We remark that such $\widetilde{Q}$ and $\Lambda$ exist for a given quiver $Q$ (cf. \cite{rupel}).

Let $\mathcal C_{\widetilde{Q}}$ be the cluster category of $k
\widetilde{Q}$, i.e., the orbit category of the bounded derived category
$\mathcal{D}^b(k\widetilde{Q})$ under the action of  the functor
$F=\tau^{-1}\circ[1]$ (cf. \cite{BMRRT}). For each $1\leq i \leq
m$, let $S_i$ be the simple $k \widetilde{Q}$-module corresponding to the vertex $i$, and let $P_i$ be the projective cover of $S_i$.  Then the indecomposable $k \widetilde{Q}$-modules and all $P_i[1]$
exhaust all indecomposable objects of $\mathcal C_{\widetilde{Q}}$. Each object $X$ in
$\mathcal C_{\widetilde{Q}}$ can be uniquely decomposed as
$$X=M\oplus P[1]$$
where $M$ is a $k \widetilde{Q}$-module and $P$ is a projective $k \widetilde{Q}$-module.

In what follows, we adopt the convention that for each given module we will use the corresponding lowercase boldface letter to denote its dimension vector. For example, given a $kQ$-module $X$, its dimension vector is denoted by ${\bf x}$, and the dimension vector of $X$ viewed as a $k\widetilde{Q}$-module is also denoted by ${\bf x}$ if there is no confusion. Thus, for each $kQ$-module $X$ we have that $(\widetilde{I}-\widetilde{R}){\bf x}=(I_m-R(\widetilde{Q})){\bf x}$.

The quantum Caldero--Chapoton map associated to an acyclic quiver
$Q$ has been defined in \cite{rupel} and \cite{fanqin}.
In \cite{rupel}, the author defined the quantum Caldero--Chapoton map
for $kQ$-modules while in \cite{fanqin} for
coefficient-free rigid objects in $\mathcal C_{\widetilde{Q}}$.
For our purpose, we need to modify  the definition as follows:

Let $M$ be a $kQ$-module and $P$ be a projective $k\widetilde{Q}$-module. We define the \emph{quantum cluster character}
\begin{equation}X_{M\oplus P[1]}=\sum_{\bf{e}} |\mathrm{Gr}_{\bf{e}} M|q^{\frac{1}{2}
\langle
\bf{p}-\bf{e},\bf{m}-\bf{e}\rangle}X^{-\widetilde{B}{\bf e}-(\widetilde{I}-\widetilde{R}'){\bf m}+{\bf t}_P},\end{equation}
where ${\bf t}_P=\Dim (P/\rad P)$ and
$\mathrm{Gr}_{\bf{e}}M$ denotes the set of all submodules $V$
of $M$ with $\Dim V= \bf{e}$. In particular, $X_{P[1]}=X^{{\bf t}_P}$.

Given a finite acyclic quiver $Q$, denote by $\mathcal{AH}_{q}(Q)$ the
$\mathbb{ZP}$-subalgebra of $\mathcal{F}_q$ generated by
all the quantum cluster characters $X_{M\oplus P[1]}$
with $M$ being a $k Q$-module and $P$ being a projective $k
\widetilde{Q}$-module, where we appoint that $P$ can be taken to be zero.

The following lemma is needed for the proof of the subsequent theorem.
\begin{lemma}\label{AE}
For any $kQ$-module $M$ and projective $k\widetilde{Q}$-module $P$, we have that
\begin{equation}
\Lambda(\widetilde{B}{\bf m},{\bf t}_P)=\lr{{\bf p},\bf{m}}.
\end{equation}
\begin{proof}
Noting that ${\bf t}_P=(I_{m}-R(\widetilde{Q})){\bf p}$, by definition, we have that
\begin{flalign*}\Lambda(\widetilde{B}{\bf m},{\bf t}_P)&=\Lambda(\widetilde{B}{\bf m},(I_{m}-R(\widetilde{Q})){\bf p})\\
&={\bf m}^{tr}{\widetilde{B}}^{tr}\Lambda(I_{m}-R(\widetilde{Q})){\bf p}\\
&={\bf m}^{tr}(I_n,0)(I_{m}-R(\widetilde{Q})){\bf p}\\
&={\bf p}^{tr}(I_{m}-R(\widetilde{Q})^{tr}){I_n\choose0
}{\bf m}\\
&={\bf p}^{tr}(I_{m}-R(\widetilde{Q})^{tr}){\bf m}\\
&=\lr{{\bf p},{\bf m}},\end{flalign*}
where we should be reminded that ${I_n\choose0}
{\bf m}=\widetilde{I}{\bf m}$ is the dimension vector of $M$ viewed as a $k\widetilde{Q}$-module, and it is still denoted by ${\bf m}$.
\end{proof}
\end{lemma}
\begin{theorem}\label{exchange2}
Let $M$ be any $kQ$-module  and $P$ any projective
$k\widetilde{Q}$-module. Then
\vspace{0.2cm}
\begin{itemize}
  \item [(1)] $\ X_{M}X_{P[1]}=q^{-\frac{1}{2}\Lambda((\widetilde{I}-\widetilde{R}){\bf m},{\bf t}_P)}X_{M\oplus
P[1]};$
\vspace{0.4cm}
\item [(2)] $\ X_{P[1]}X_{M}=q^{\frac{1}{2}\Lambda((\widetilde{I}-\widetilde{R}){\bf m},{\bf t}_P)-\lr{{\bf p},{\bf m}}}
\sum\limits_{[B],[Q]}|{}_Q\Hom_{k\widetilde{Q}}(P,M)_B| X_{B\oplus Q[1]}.$
\end{itemize}
\end{theorem}

\begin{proof}
$(1)$~~By definition,
\begin{flalign*}
   X_{M}X_{P[1]}\nonumber
   &=\sum_{\bf{e}} |\mathrm{Gr}_{\bf{e}} M|q^{-\frac{1}{2}\langle
\bf{e},\bf{m}-\bf{e}\rangle}X^{-\widetilde{B}{\bf e}-(\widetilde{I}-\widetilde{R}')\bf{m}} X^{{\bf t}_P}\\
  &=\sum_{\bf{e}} |\mathrm{Gr}_{\bf{e}} M|q^{-\frac{1}{2}
\langle{\bf e},\bf{m}-\bf{e}\rangle}q^{\frac{1}{2}\Lambda(-\widetilde{B}{\bf e}
-(\widetilde{I}-\widetilde{R}'){\bf m},{\bf t}_P)}X^{-\widetilde{B}{\bf e}-(\widetilde{I}-\widetilde{R}'){\bf m}+{\bf t}_P}.
\end{flalign*}
Noting that
\begin{flalign*}
\Lambda(-\widetilde{B}{\bf e}
-(\widetilde{I}-\widetilde{R}'){\bf m},{\bf t}_P)
   &=\Lambda(-\widetilde{B}{\bf e}+\widetilde{B}{\bf m}
-(\widetilde{I}-\widetilde{R}){\bf m},{\bf t}_P)\nonumber\\
&=\Lambda(-(\widetilde{I}-\widetilde{R}){\bf m},{\bf t}_P)+\Lambda(\widetilde{B}({\bf m}-{\bf e}),{\bf t}_P)\\
 &=-\Lambda((\widetilde{I}-\widetilde{R}){\bf m},{\bf t}_P)+\langle \bf{p},\bf{m}-\bf{e}\rangle,
\end{flalign*}
we obtain that
\begin{flalign*}
X_{M}X_{P[1]}
&=q^{-\frac{1}{2}\Lambda((\widetilde{I}-\widetilde{R}){\bf m},{\bf t}_P)}\sum_{\bf{e}} |\mathrm{Gr}_{\bf{e}} M|q^{-\frac{1}{2}\langle
\bf{e},\bf{m}-\bf{e}\rangle}q^{\frac{1}{2}\langle
\bf{p},\bf{m}-\bf{e}\rangle}X^{-\widetilde{B}{\bf e}-(\widetilde{I}-\widetilde{R}'){\bf m}+{\bf t}_P}\\
&=q^{-\frac{1}{2}\Lambda((\widetilde{I}-\widetilde{R}){\bf m},{\bf t}_P)}X_{M\oplus
P[1]}.
\end{flalign*}

$(2)$~~On the one hand, by definition,
\begin{flalign*}
X_{P[1]}X_{M}&=\sum_{[H],[G]}q^{-\frac{1}{2}\langle {\bf h},{\bf g}\rangle}F^{M}_{GH}X^{{\bf t}_P}X^{-\widetilde{B}{\bf h}-(\widetilde{I}-\widetilde{R}')\bf{m}}\\
&=\sum_{[H],[G]}q^{-\frac{1}{2}\langle {\bf h},{\bf g}\rangle}F^{M}_{GH}q^{\frac{1}{2}\Lambda({\bf t}_P,-\widetilde{B}{\bf h}-(\widetilde{I}-\widetilde{R}')\bf{m})}X^{-\widetilde{B}{\bf h}-(\widetilde{I}-\widetilde{R}'){\bf m}
  +{\bf t}_P}\\
&=q^{\frac{1}{2}\Lambda((\widetilde{I}-\widetilde{R}){\bf m},{\bf t}_P)}\sum_{[H],[G]}q^{-\frac{1}{2}\langle
{\bf p}+{\bf h},{\bf g}\rangle}F^{M}_{GH}X^{-\widetilde{B}{\bf h}-(\widetilde{I}-\widetilde{R}'){\bf m}+{\bf t}_P}.
\end{flalign*}

On the other hand, by definition,
\begin{flalign*}&\sum_{[B],[Q]}|{}_Q\Hom_{k\widetilde{Q}}(P,M)_B|X_{B\oplus
Q[1]}\\&=\sum_{[B],[Q],[L],[G],[Y]}a_LF^{P}_{LQ}F^{M}_{BL}q^{\frac{1}{2}\langle
{\bf q}-{\bf y},{\bf g}\rangle}F^{B}_{GY}X^{-\widetilde{B}{\bf y}-(\widetilde{I}-\widetilde{R}'){\bf b}+{\bf t}_Q}\\
&=\sum_{[H],[Q],[L],[G],[Y]}q^{\frac{1}{2}\langle
{\bf q}-{\bf y},{\bf g}\rangle}a_LF^{P}_{LQ}F^{H}_{YL}F^{M}_{GH}X^{-\widetilde{B}{\bf y}-(\widetilde{I}-\widetilde{R}')\bf{b}+{\bf t}_Q}\\
&=\sum_{[H],[Q],[G],[Y]}q^{\frac{1}{2}\langle
{\bf q}-{\bf y},{\bf g}\rangle}|{}_Q\Hom_{k\widetilde{Q}}(P,H)_Y|F^{M}_{GH}X^{-\widetilde{B}{\bf y}-(\widetilde{I}-\widetilde{R}')\bf{b}+{\bf t}_Q}
\end{flalign*}where we have used the associativity formula (\ref{jiehe}) for the second equation.

Noting that
$$\bf{b}+\bf{l}=\bf{m},~~\bf{q}+\bf{l}=\bf{p},~~\bf{y}+\bf{l}=\bf{h},~~\text{and~~thus}~~\bf{q}-\bf{y}=\bf{p}-\bf{h},$$
we have that
\begin{flalign*}
&-\widetilde{B}{\bf y}-(\widetilde{I}-\widetilde{R}'){\bf b}+{\bf t}_Q\\
&=-\widetilde{B}{\bf h}+\widetilde{B}{\bf l}-(\widetilde{I}-\widetilde{R}'){\bf b}+(I_{m}-R(\widetilde{Q})){\bf q}\\
&=-\widetilde{B}{\bf h}+\widetilde{R}'{\bf l}-\widetilde{R}{\bf l}
-\widetilde{I}{\bf b}+\widetilde{R}'{\bf b}+(I_{m}-R(\widetilde{Q})){\bf q}\\
&=-\widetilde{B}{\bf h}+\widetilde{R}'{\bf m}-\widetilde{R}{\bf l}
-\widetilde{I}({\bf m}-{\bf l})+(I_{m}-R(\widetilde{Q})){\bf q}\\
&=-\widetilde{B}{\bf h}-(\widetilde{I}-\widetilde{R}'){\bf m}
  +(\widetilde{I}-\widetilde{R}){\bf l}+(I_{m}-R(\widetilde{Q})){\bf q}\\
&=-\widetilde{B}{\bf h}-(\widetilde{I}-\widetilde{R}'){\bf m}
  +(I_{m}-R(\widetilde{Q})){\bf l}+(I_{m}-R(\widetilde{Q})){\bf q}\\
&=-\widetilde{B}{\bf h}-(\widetilde{I}-\widetilde{R}'){\bf m}
  +(I_{m}-R(\widetilde{Q})){\bf p}\\
  &=-\widetilde{B}{\bf h}-(\widetilde{I}-\widetilde{R}'){\bf m}
  +{\bf t}_P.
\end{flalign*}
Hence,
\begin{flalign*}
&\sum_{[B],[Q]}|{}_Q\Hom_{k\widetilde{Q}}(P,M)_B|X_{B\oplus
Q[1]}\\
&=\sum_{[H],[Q],[G],[Y]}q^{\frac{1}{2}\langle
{\bf p}-{\bf h},{\bf g}\rangle}|{}_Q\Hom_{k\widetilde{Q}}(P,H)_Y|F^{M}_{GH}X^{-\widetilde{B}{\bf h}-(\widetilde{I}-\widetilde{R}'){\bf m}
  +{\bf t}_P}\\
  &=\sum_{[H],[G]}q^{\frac{1}{2}\langle
{\bf p}-{\bf h},{\bf g}\rangle+\lr{{\bf p},{\bf h}}}F^{M}_{GH}X^{-\widetilde{B}{\bf h}-(\widetilde{I}-\widetilde{R}'){\bf m}
  +{\bf t}_P}\\
&=q^{\lr{{\bf p},{\bf m}}}\sum_{[H],[G]}q^{-\frac{1}{2}\lr{{\bf p}+{\bf h},{\bf g}}}F^{M}_{GH}X^{-\widetilde{B}{\bf h}-(\widetilde{I}-\widetilde{R}'){\bf m}
  +{\bf t}_P}.
\end{flalign*}
\end{proof}

\section{Quantum cluster algebras via Hall algebras}
Let $Q$ be an acyclic quiver with the vertex set $\{1,2,\cdots,n\}$. As stated in \cite{Rup}, a result of Fomin and Zelevinsky in \cite{FZ4} asserts that the cluster variables
are completely determined by the cluster variables of the principal coefficient quantum cluster algebra.
So we consider the quiver $\widetilde{Q}$ associated to $Q$ as follows: add the vertices $\{n+1,\dots,2n\}$ to the quiver $Q$ with the additional arrows
$n+i\longrightarrow i$ for any $1\leq i\leq n$. By \cite{BZ05}, we can take a skew-symmetric $2n\times 2n$ integral
matrix $\Lambda$ such that
\begin{align}\label{eq:simply_laced_compatible2}
\Lambda(-\widetilde{B})={I_n\choose0}.
\end{align}

In order to relate the Hall algebra $\M\H(\A_Q)$ to the algebra $\mathcal{AH}_{q}(Q)$, we twist the multiplication on $\M\H(\A_Q)$, and define $\mathcal{MH}_{\Lambda}(\A_Q)$ to be the same module as $\mathcal{MH}(\A_Q)$ but with the twisted
multiplication defined on basis elements by
{\begin{equation}\begin{split}
   &(K_{\alpha}\ast\mathbb{X}_{M\oplus P[1]})\star (K_{\beta}\ast\mathbb{X}_{N\oplus Q[1]})=\\&
   q^{\frac{1}{2}\Lambda((\widetilde{I}-\widetilde{R})({\bf m}-{\bf p})-\widetilde{\alpha},(\widetilde{I}-\widetilde{R})(\bf{n}
   -\bf{q})-\widetilde{\beta})}
   (K_{\alpha}\ast\mathbb{X}_{M\oplus P[1]})\ast (K_{\beta}\ast\mathbb{X}_{N\oplus Q[1]}),\end{split}
\end{equation}}
where $M, N\in\A_Q$, $P, Q\in\P$, $\alpha,\beta\in \mathbb{Z}^{n}$, and for each $\alpha\in \mathbb{Z}^{n}$, $\widetilde{\alpha}:={0\choose\alpha}
\in \mathbb{Z}^{2n}$.

For any $\alpha\in \mathbb{Z}^{n}$, $M\in\A_Q$ and $P\in\P$, we define
$$\grad(K_{\alpha}\ast\mathbb{X}_{M\oplus P[1]}):=(\widetilde{I}-\widetilde{R})({\bf m}-{\bf p})-\widetilde{\alpha},$$
and for each $\varpi\in\mathbb{Z}^{2n}$ we define $\mathcal{MH}(\A_Q)_\varpi$ to be the submodule of $\mathcal{MH}(\A_Q)$ spanned by the elements $K_{\alpha}\ast\mathbb{X}_{M\oplus P[1]}$ with
$\grad(K_{\alpha}\ast\mathbb{X}_{M\oplus P[1]})=\varpi$.

The following lemma shows that the functions $\grad$ provide a $\mathbb{Z}^{2n}$-grading on the Hall algebra $\mathcal{MH}(\A_Q)$.
\begin{lemma}
For any $\varpi_1,\varpi_2\in\mathbb{Z}^{2n}$, we have that
\begin{equation*}\mathcal{MH}(\A_Q)_{\varpi_1}\ast\mathcal{MH}(\A_Q)_{\varpi_2}\subseteq\mathcal{MH}(\A_Q)_{\varpi_1+\varpi_2}.\end{equation*}
That is, $\mathcal{MH}(\A_Q)$ is a $\mathbb{Z}^{2n}$-graded algebra.
\end{lemma}
\begin{proof}
Let $M, N\in\A_Q$, $P, Q\in\P$, $\alpha,\beta\in \mathbb{Z}^{n}$ such that $\grad(K_{\alpha}\ast\mathbb{X}_{M\oplus P[1]})=\varpi_1$ and $\grad(K_{\beta}\ast\mathbb{X}_{N\oplus Q[1]})=\varpi_2$. Then
\begin{flalign*}
&(K_{\alpha}\ast\mathbb{X}_{M\oplus P[1]})\ast (K_{\beta}\ast\mathbb{X}_{N\oplus Q[1]})\\&=
K_{\alpha}\ast\mathbb{X}_{M}\ast\X_{P[1]}\ast K_{\beta}\ast\mathbb{X}_{N}\ast\X_{Q[1]}\\
&=K_{\alpha+\beta}\ast\mathbb{X}_{M}\ast\X_{P[1]} \ast\mathbb{X}_{N}\ast\X_{Q[1]}\\
&=\sum\limits_{[B],[Q']}q^{-\lr{P,N}}|{}_{Q'}\Hom_{\A}(P,N)_B|K_{\alpha+\beta}\ast\X_M\ast\X_B\ast\X_{(Q'\oplus Q)[1]}\\
&=\sum\limits_{[B],[Q'],[L]}q^{-\lr{P,N}+\lr{M,B}}|{}_{Q'}\Hom_{\A}(P,N)_B|\cdot\frac{|\mathrm{Ext}_{\A}^{1}(M,B)_{L}|}{|\mathrm{Hom}_{\A}(M,B)|}K_{\alpha+\beta}\ast\X_{L\oplus(Q'\oplus Q)[1]}.
\end{flalign*}
For each $K_{\alpha+\beta}\ast\X_{L\oplus(Q'\oplus Q)[1]}$ in the summation above, since ${\bf l}={\bf m}+{\bf b}={\bf m}+{\bf q'}-{\bf p}+{\bf n}$, we obtain that \begin{flalign*}\grad(K_{\alpha+\beta}\ast\X_{L\oplus(Q'\oplus Q)[1]})&=(\widetilde{I}-\widetilde{R})({\bf l}-({\bf q'}+{\bf q}))-(\widetilde{\alpha}+\widetilde{\beta})\\&=(\widetilde{I}-\widetilde{R})({\bf m}+{\bf n}-{\bf p}-{\bf q})-(\widetilde{\alpha}+\widetilde{\beta})\\
&=\grad(K_{\alpha}\ast\mathbb{X}_{M\oplus P[1]})+\grad (K_{\beta}\ast\mathbb{X}_{N\oplus Q[1]}).
\end{flalign*}
\end{proof}

\begin{lemma}
The algebra $\mathcal{MH}_{\Lambda}(\A_Q)$ is still associative.
\end{lemma}
\begin{proof}
For any $L,M,N\in\A$, $P,Q,R\in\P$ and $\alpha,\beta,\gamma\in\mathbb{Z}^{n}$, it is easy to see that
\begin{flalign*}&((K_{\alpha}\ast\mathbb{X}_{M\oplus P[1]})\star (K_{\beta}\ast\mathbb{X}_{N\oplus Q[1]}))\star (K_{\gamma}\ast\mathbb{X}_{L\oplus R[1]})\\&=
q^{\frac{1}{2}e}(K_{\alpha}\ast\mathbb{X}_{M\oplus P[1]})\ast (K_{\beta}\ast\mathbb{X}_{N\oplus Q[1]})\ast (K_{\gamma}\ast\mathbb{X}_{L\oplus R[1]})\\&
=(K_{\alpha}\ast\mathbb{X}_{M\oplus P[1]})\star ((K_{\beta}\ast\mathbb{X}_{N\oplus Q[1]})\star (K_{\gamma}\ast\mathbb{X}_{L\oplus R[1]}))\end{flalign*}
where $e=\Lambda((\widetilde{I}-\widetilde{R})({\bf m}-{\bf p})-\widetilde{\alpha},(\widetilde{I}-\widetilde{R})({\bf n}
-{\bf q})-\widetilde{\beta})+\Lambda((\widetilde{I}-\widetilde{R})({\bf m}-{\bf p})-\widetilde{\alpha},(\widetilde{I}-\widetilde{R})
({\bf l}-{\bf r})-\widetilde{\gamma})+\Lambda((\widetilde{I}-\widetilde{R})({\bf n}-{\bf q})-\widetilde{\beta},
(\widetilde{I}-\widetilde{R})({\bf l}
-{\bf r})-\widetilde{\gamma})$. That is, the algebra $\mathcal{MH}_{\Lambda}(\A_Q)$ is associative.\end{proof}

Using Theorem \ref{ydygx}, we write down the defining relations of $\mathcal{MH}_{\Lambda}(\A_Q)$ in the following
\begin{proposition}
The algebra $\mathcal{MH}_{\Lambda}(\A_Q)$ is generated by all $K_{\alpha}$ and $\mathbb{X}_{M\oplus P[1]}$ $($with $\alpha\in\mathbb{Z}^{n}$, $M\in\A_Q$ and $P\in\P$$)$,
which are subject to the following relations
\begin{equation}\label{x1}
\begin{split}K_{\alpha}\star K_{\beta}=q^{\frac{1}{2}\Lambda(\widetilde{\alpha},\widetilde{\beta})}K_{\alpha+\beta}=q^{\Lambda(\widetilde{\alpha},\widetilde{\beta})}K_{\beta}\star K_{\alpha};\end{split}\end{equation}
\begin{equation}\label{x2}
\begin{split}K_{\alpha}\star\mathbb{X}_{M\oplus P[1]}
    =q^{-\Lambda(\widetilde{\alpha},(\widetilde{I}-\widetilde{R})({\bf m}-{\bf p}))}\mathbb{X}_{M\oplus P[1]}\star K_{\alpha};\end{split}\end{equation}
\begin{equation}\label{x3}
\begin{split}
\mathbb{X}_{P[1]}\star \mathbb{X}_{Q[1]}&=
q^{\frac{1}{2}\Lambda((\widetilde{I}-\widetilde{R}){\bf p},(\widetilde{I}-\widetilde{R})\bf{q}
)}\mathbb{X}_{(P\oplus Q)[1]}\\
&=q^{\Lambda((\widetilde{I}-\widetilde{R}){\bf p},(\widetilde{I}-\widetilde{R})\bf{q})}\mathbb{X}_{Q[1]}\star\mathbb{X}_{P[1]};
\end{split}\end{equation}
\begin{equation}\label{x4}\mathbb{X}_{M}\star\mathbb{X}_{N}=q^{\frac{1}{2}\Lambda((\widetilde{I}-\widetilde{R}){\bf m},
(\widetilde{I}-\widetilde{R})\bf{n})+\langle\bf{m},\bf{n}\rangle}\sum_{[L]}\frac{|\mathrm{Ext}_{\A}^{1}(M,N)_{L}|}{|\mathrm{Hom}_{\A}(M,N)|}\mathbb{X}_L;\end{equation}
\begin{equation}\label{x5}\mathbb{X}_{M}\star\mathbb{X}_{P[1]}=q^{-\frac{1}{2}\Lambda((\widetilde{I}-\widetilde{R}){\bf m},(\widetilde{I}-\widetilde{R}){\bf p})}\mathbb{X}_{M\oplus
P[1]};\end{equation}
\begin{equation}\label{x6}\mathbb{X}_{P[1]}\star\mathbb{X}_{M}=q^{-\frac{1}{2}\Lambda((\widetilde{I}-\widetilde{R}){\bf p},(\widetilde{I}-\widetilde{R}){\bf m})-\lr{{\bf p},{\bf m}}}
\sum\limits_{[B],[Q]}|{}_Q\Hom_{\A}(P,M)_B|\mathbb{X}_{B\oplus Q[1]};\end{equation}
for any $\alpha,\beta\in\mathbb{Z}^n$, $M,N\in\A_Q$ and $P,Q\in\P$.
\end{proposition}
\begin{theorem}\label{alg-homo}
There exists a surjective algebra homomorphism $$\xymatrix{\Psi: \mathcal{MH}_{\Lambda}(\A_Q)\ar@{->>}[r]& \mathcal{AH}_{q}(Q)}$$ defined on generators by
$$K_{\alpha}\mapsto X^{\widetilde{\alpha}}\quad\text{and}\quad\mathbb{X}_{M\oplus P[1]}\mapsto X_{M\oplus P[1]}$$
for any $\alpha\in\mathbb{Z}^n$, $M\in\A_Q$ and $P\in\P$.
\end{theorem}
\begin{proof}
In order to prove $\Psi$ is a homomorphism of algebras, it suffices to prove that $\Psi$ preserves the relations $(\ref{x1}$-$\ref{x6})$. Clearly, the relation $(\ref{x1})$ is preserved.

By definition,
\begin{flalign*}
X^{\widetilde{\alpha}}X_{M\oplus P[1]}&=\sum\limits_{[H],[G]}q^{\frac{1}{2}\lr{{\bf p}-{\bf h},{\bf g}}}F_{GH}^MX^{\widetilde{\alpha}}X^{-\widetilde{B}{\bf h}-(\widetilde{I}-\widetilde{R}'){\bf m}+{\bf t}_P}\\
&=\sum\limits_{[H],[G]}q^{\frac{1}{2}\lr{{\bf p}-{\bf h},{\bf g}}}F_{GH}^Mq^{\Lambda(\widetilde{\alpha},-\widetilde{B}{\bf h}-(\widetilde{I}-\widetilde{R}'){\bf m}+{\bf t}_P)}X^{-\widetilde{B}{\bf h}-(\widetilde{I}-\widetilde{R}'){\bf m}+{\bf t}_P}X^{\widetilde{\alpha}}
\end{flalign*}
and
\begin{flalign*}
\Lambda(\widetilde{\alpha},-\widetilde{B}{\bf h}-(\widetilde{I}-\widetilde{R}'){\bf m}+{\bf t}_P)&=\Lambda(\widetilde{\alpha},\widetilde{B}({{\bf m}-{\bf h}})-(\widetilde{I}-\widetilde{R}){\bf m}+{\bf t}_P)\\
&=-\Lambda(\widetilde{B}{\bf g},\widetilde{\alpha})-\Lambda(\widetilde{\alpha},(\widetilde{I}-\widetilde{R})({\bf m}-{\bf p})).
\end{flalign*}
We claim that $\Lambda(\widetilde{B}{\bf g},\widetilde{\alpha})=0$. In fact, for any $\alpha=(\alpha_1,\cdots,\alpha_n)\in\mathbb{Z}^n$, $\widetilde{\alpha}=\sum\limits_{i=1}^n\alpha_i{\bf t}_{P_{n+i}}$.
Then by Lemma \ref{AE},
\begin{flalign*}
\Lambda(\widetilde{B}{\bf g},\widetilde{\alpha})=\sum\limits_{i=1}^n\alpha_i\Lambda(\widetilde{B}{\bf g},{\bf t}_{P_{n+i}})
=\sum\limits_{i=1}^n\alpha_i\lr{{\bf p}_{n+i},{\bf g}}=0.
\end{flalign*}
Hence, \begin{equation*}X^{\widetilde{\alpha}}X_{M\oplus P[1]}=q^{-\Lambda(\widetilde{\alpha},(\widetilde{I}-\widetilde{R})({\bf m}-{\bf p}))}X_{M\oplus P[1]}X^{\widetilde{\alpha}}.\end{equation*} That is, the relation $(\ref{x2})$ is preserved.

By definition,
\begin{flalign*}
{X}_{P[1]}{X}_{Q[1]}=X^{{\bf t}_P}X^{{\bf t}_Q}&=q^{\frac{1}{2}\Lambda((\widetilde{I}-\widetilde{R}){\bf p},(\widetilde{I}-\widetilde{R})\bf{q}
)}{X}^{{\bf t}_P+{\bf t}_Q}\\&=
q^{\frac{1}{2}\Lambda((\widetilde{I}-\widetilde{R}){\bf p},(\widetilde{I}-\widetilde{R})\bf{q}
)}{X}_{(P\oplus Q)[1]}
\end{flalign*}
and
\begin{flalign*}
{X}_{P[1]}{X}_{Q[1]}=X^{{\bf t}_P}X^{{\bf t}_Q}&=q^{\Lambda((\widetilde{I}-\widetilde{R}){\bf p},(\widetilde{I}-\widetilde{R})\bf{q}
)}X^{{\bf t}_Q}X^{{\bf t}_P}\\
&=q^{\Lambda((\widetilde{I}-\widetilde{R}){\bf p},(\widetilde{I}-\widetilde{R})\bf{q}
)}{X}_{Q[1]}{X}_{P[1]}.\end{flalign*}Thus, the relation $(\ref{x3})$ is preserved.

On the one hand, the following equation has been proved in \cite{DSC,DX}
\begin{equation}{X}_{M}{X}_{N}=q^{\frac{1}{2}\Lambda((\widetilde{I}-\widetilde{R}'){\bf m},
(\widetilde{I}-\widetilde{R}')\bf{n})+\langle\bf{m},\bf{n}\rangle}\sum_{[L]}\frac{|\mathrm{Ext}_{\A}^{1}(M,N)_{L}|}{|\mathrm{Hom}_{\A}(M,N)|}X_L.\end{equation}
On the other hand, by \cite[Lemma 3.1]{DSC},
\begin{flalign*}
&\Lambda((\widetilde{I}-\widetilde{R}'){\bf m},
(\widetilde{I}-\widetilde{R}')\bf{n})=\\&\Lambda((\widetilde{I}-\widetilde{R}){\bf m},
(\widetilde{I}-\widetilde{R}){\bf n})-\Lambda((\widetilde{I}-\widetilde{R}'){\bf m},\widetilde{B}{\bf n})-\Lambda(\widetilde{B}{\bf m},(\widetilde{I}-\widetilde{R}'){\bf n})-\Lambda(\widetilde{B}{\bf m},\widetilde{B}{\bf n})\\
&=\Lambda((\widetilde{I}-\widetilde{R}){\bf m},
(\widetilde{I}-\widetilde{R}){\bf n})+\lr{{\bf n},{\bf m}}-\lr{{\bf m},{\bf n}}+\lr{{\bf m},{\bf n}}-\lr{{\bf n},{\bf m}}\\
&=\Lambda((\widetilde{I}-\widetilde{R}){\bf m},
(\widetilde{I}-\widetilde{R}){\bf n}).
\end{flalign*}
Thus, the relation $(\ref{x4})$ is preserved.

By Theorem \ref{exchange2}, we know that the relations $(\ref{x5}$-$\ref{x6})$ are preserved. Hence, $\Psi$ is a homomorphism of algebras.

For each $1\leq i\leq n$, $X_{P_{n+i}[1]}=X^{{\bf t}_{P_{n+i}}}=X^{e_{n+i}}$, and thus $\Psi(K_{\hat{S}_i})=X_{P_{n+i}[1]}$. Then it is clear that $\Psi$ is surjective. Therefore, we complete the proof.
\end{proof}

Let $\A_q(Q)$ be the quantum cluster algebra with principal coefficients corresponding to the quiver $Q$, which is the subalgebra of $\mathcal{AH}_{q}(Q)$ generated by
\begin{equation*}
\{X^{\widetilde{\alpha}}, X_{M\oplus P[1]}~|~\alpha\in\mathbb{Z}^n, M\in\Ind(\A_Q)~~\text{is rigid and}~~P\in\Ind(\P)\}.
\end{equation*}
Parallelly, we define $\mathcal{MC}_{\Lambda}(\A_Q)$ to be the subalgebra of $\mathcal{MH}_{\Lambda}(\A_Q)$ generated by $$\{K_{\alpha},\X_{M\oplus P[1]}~|~\alpha\in\mathbb{Z}^n, M\in\Ind(\A_Q)~~\text{is rigid and}~~P\in\Ind(\P)\}.$$
Then we have the following

\begin{corollary}
There exists a surjective algebra homomorphism $$\xymatrix{\psi: \mathcal{MC}_{\Lambda}(\A_Q)\ar@{->>}[r]& \mathcal{A}_{q}(Q).}$$\end{corollary}
\begin{proof}
Taking $\psi$ to be the restriction of $\Psi$ in Theorem \ref{alg-homo} to $\mathcal{MC}_{\Lambda}(\A_Q)$ gives the proof.
\end{proof}
\begin{remark}
By Proposition \ref{generators},
we know that the set $$S:=\{K_{\alpha},\mathbb{X}_M,\mathbb{X}_{P[1]}~|~\alpha\in\mathbb{Z}^{n},~M\in\Ind(\A_Q)~~\text{and}~~P\in\Ind(\P)\}$$ is a generating set of the algebra $\mathcal{MH}_{\Lambda}(\A_Q)$. For a Dynkin quiver $Q$, each indecomposable $kQ$-module is rigid. Thus, in this case, $\mathcal{MC}_{\Lambda}(\A_Q)=\mathcal{MH}_{\Lambda}(\A_Q)$. By Theorem \ref{alg-homo}, the images of all the elements in $S$ under the map $\Psi$, which constitute the set $\{X^{\widetilde{\alpha}},{X}_M,{X}_{P[1]}~|~\alpha\in\mathbb{Z}^{n},~M\in\Ind(\A_Q)~~\text{and}~~P\in\Ind(\P)\}$, generate $\mathcal{AH}_{q}(Q)$.
Also, for a Dynkin quiver $Q$, since each indecomposable $kQ$-module is rigid, we have that $\mathcal{AH}_{q}(Q)=\mathcal{A}_{q}(Q)$.
\end{remark}

\begin{corollary}
Let $Q$ be a Dynkin quiver.
There exists a surjective algebra homomorphism $$\xymatrix{\psi: \mathcal{MH}_{\Lambda}(\A_Q)\ar@{->>}[r]& \mathcal{A}_{q}(Q).}$$\end{corollary}

\section*{Acknowledgments}
The authors are grateful to the anonymous referees for their valuable suggestions and comments,
and supported partially by the National Natural Science Foundation of China (No.s 11771217, 11471177, 11801273), Natural Science Foundation of Jiangsu Province of China (No.BK20180722)
and Natural Science Foundation of Jiangsu Higher Education Institutions of China  (No.18KJB110017).

\end{document}